\documentclass[a4paper,10pt,reqno]{amsart}
\usepackage{amssymb,latexsym,bbm,amsmath,amsfonts}
\usepackage{amsthm}
\usepackage[mathscr]{eucal}
\usepackage{rotating}
\usepackage{multirow}
\usepackage{slashbox}

\numberwithin{equation}{section}
\newtheorem{theorem}{Theorem}
\newtheorem{lemma}{Lemma}
\newtheorem{proposition}{Proposition}
\newtheorem{corollary}{Corollary}

\theoremstyle{definition}
\newtheorem{definition}[theorem]{Definition}

\newtheorem{example}[theorem]{Example}

\theoremstyle{remark}
\newtheorem{remark}{Remark}

\newcommand{\norm}[1]{\left\lVert #1 \right\rVert}
\newcommand{\abs}[1]{\left\lvert #1 \right\rvert}    

\DeclareMathOperator{\R}{{\mathbb R}}                
\DeclareMathOperator{\Rp}{{\mathbb R}_+}             
\DeclareMathOperator{\C}{{\mathbb C}}                
\DeclareMathOperator{\I}{I}                          %
\renewcommand{\P}{P}

\newcommand{\e}{\mathrm{e}}                                                  

\begin{document}

\title[Asymptotics of Affine Stochastic Volterra Equations]
{
Exact Pathwise and Mean--Square Asymptotic Behaviour of Stochastic
Affine Volterra and Functional Differential Equations}

\author{John~A.~D.~Appleby}
\address{Edgeworth Centre for Financial Mathematics, School of Mathematical
Sciences, Dublin City University, Glasnevin, Dublin 9, Ireland}
\email{john.appleby@dcu.ie} \urladdr{webpages.dcu.ie/\textasciitilde
applebyj}

\author{John~A.~Daniels}
\address{Edgeworth Centre for Financial Mathematics, School of Mathematical
Sciences, Dublin City University, Glasnevin, Dublin 9, Ireland}
\email{john.daniels2@mail.dcu.ie}


\thanks{Both authors gratefully acknowledge Science Foundation Ireland for the support of this research
under the Mathematics Initiative 2007 grant 07/MI/008 ``Edgeworth
Centre for Financial Mathematics''. The second author is also
supported by the Irish Research Council for Science, Engineering and
Technology under the Embark Initiative grant.}
\subjclass[2010]{Primary: 47G10, 45A05,
45D05, 34K06, 34K25, 34K50,
60H10, 60H20, 45F99, 45P05, 45J05, 45M05, 60G15;
Secondary: 91B70, 91G80} \keywords{Stochastic linear operator,
admissibility, Volterra operator, almost sure convergence, mean
square convergence, stochastic functional differential equations,
stochastic Volterra equations, differential resolvent,
integro--differential equations, linear Volterra equation,
characteristic equation, characteristic exponent}
\date{24 October 2012}

\begin{abstract}
The almost sure rate of exponential-polynomial growth or decay of
affine stochastic Volterra and affine stochastic finite-delay
equations is investigated. These results are achieved under suitable
smallness conditions on the intensities of the deterministic and
stochastic perturbations diffusion, given that the asymptotic
behaviour of the underlying deterministic resolvent is determined by
the zeros of its characteristic equation. The results rely heavily
upon a stochastic variant of the admissibility theory for linear
Volterra operators.
\end{abstract}

\maketitle

\section{Introduction}
Interest in stochastic functional differential equations, including
stochastic differential equations with delay, and stochastic
Volterra equations, has increased in recent years, in part because
of their attraction for modelling real--world systems in which the
change in the state of a system is both random and depends on the
path of the process in the past. Examples include population biology
(Mao~\cite{XM:2005}, Mao and Rassias~\cite{MR:2005, MR:2007}),
neural networks (cf. e.g. Blythe et al.~\cite{BlyXMAS:2001}),
viscoelastic materials subjected to heat or mechanical stress
Drozdov and Kolmanovskii~\cite{DrozKol:1992a}, Caraballo et
al.~\cite{CarChuesRubReal:2007a}, Mizel and
Trutzer~\cite{MizTrut:84,MizTrut:85}), or financial mathematics {Anh
et al.~\cite{anh,aik}, Appleby et al.~\cite{jamrcs10}, Appleby and
Daniels~\cite{App_Dan}, Arrojas et al.~\cite{ArrHuMoh:2007}, Hobson
and Rogers~\cite{HobRog}, and Bouchaud and
Cont~\cite{ContBouchaud:1998a}.

Naturally, in all these disciplines, there is a great interest in
understanding the long--run behaviour of solutions. In disciplines
such as engineering and physics it is often of great importance to
know that the system is \emph{stable}, in the sense that the
solution of the mathematical model converges in some sense to
equilibrium. Consequently, a great deal of mathematical activity has
been devoted to the question of stability of point equilibria of
stochastic functional differential equations and also to the rate at
which solutions converge. The literature is extensive, but a flavour
of the work can be found in the monographs of Mao~\cite{M94,Mao2},
Mohammed~\cite{Moh:84}, and Kolmanovskii and
Myskhis~\cite{KolMys:99}. Results are known concerning the
asymptotic behaviour of affine stochastic Volterra equations,
including rates of convergence (see~\cite{apprie,apprie2}), but
generally upper bounds on the solutions are found, rather than exact
rates of decay. In this paper, we investigate not only the exact
rate of convergence of solutions to point equilibria, but also the
exact rate of growth of solutions of affine equations, which are of
interest in studying the explosive growth or collapse of asset
prices in financial market models. This develops results established
in~\cite{jamrcs10}.

To determine the precise asymptotic results we require, it proves
efficient and instructive to ask first a more general question
concerning the asymptotic behaviour of stochastic integrals of the
form
\begin{equation}\label{eq:HfB}
(\mathcal{H}f)(t):=\int_0^t H(t,s)f(s)\,dB(s)
\end{equation}
where $H$ is a deterministic Volterra kernel and $f$ is a
deterministic function on $[0,\infty)$.

There is a deterministic theory of admissible operators which
enables one to give precise asymptotic information regarding the
solutions of integral and differential equations. As part of the
analysis of such theory one encounters deterministic counterparts of
\eqref{eq:HfB}. It is then unsurprising to see \eqref{eq:HfB} in the
study of affine stochastic differential equations. The admissibility
theory is often useful when any forcing terms are of the same or
smaller order to the solution of the unperturbed equation.

This theory is examined in depth in Appleby et al.~\cite{AppDanRey},
the chief results are summarised in Section~\ref{sect:stochlim}. It
is supposed in \cite{AppDanRey} that there exists a
$H_\infty:\R\to\R$ such that \eqref{eq:HfB} converges almost surely
according to
\[
\lim_{t\to\infty}\int_0^t H(t,s)f(s)\,dB(s) =
\int_{0}^{\infty}H_{\infty}(s)f(s)dB(s).
\]
This paper largely employs the admissibility theory of
\cite{AppDanRey}. However as \cite{AppDanRey} does not provide a
method of constructing such a $H_\infty$, we remark when
hypothesising the precise form of $H_\infty$ it often proves useful
as to examine $\lim_{t\to\infty} H(t,s)$. If there is any growing or
indeed oscillating component in $t\mapsto H(t,s)$ one may use this
to deduce the correct form of $H_\infty$.

Once we have developed some general results concerning the
asymptotic behaviour of $\mathcal{H}f$, the majority of the paper is
devoted to applying this theory to describe the fine structure of
the asymptotic behaviour of affine stochastic functional
differential equations of the form
\[
dX(t)=L(t,X_t)\,dt + \Sigma(t)\,dB(t)
\]
where $L=L(\phi)$ is a linear functional from $C([-\tau,0])$ to
$\mathbb{R}^d$, or $L(t,\phi_t)$ is a linear convolution Volterra
functional from $C([0,\infty))$ to $\mathbb{R}^d$. Therefore, we are
chiefly interested in the effect of time--dependent stochastic
perturbations on the asymptotic behaviour of autonomous (or
asymptotically autonomous) linear functional differential equations.
It is assumed that the asymptotic behaviour of solutions of the
underlying fundamental solution of differential resolvent can be
described in terms of the solutions of the characteristic equation,
and that such solutions lie in the region of existence of the
transform of the resolvent.

Results of Mohammed and Scheutzow~\cite{smms:1990} show that with
respect to white noise perturbations, the Liapunov spectrum of
deterministic functional differential equations is preserved, to the
extent that the leading positive Liapunov exponent of the
deterministic equation becomes the a.s. leading Liapunov exponent of
the stochastic equation. However, it is also of interest to ask
whether oscillation, or multiplicity of the characteristic equations
are preserved when the noise intensity is sufficiently small (or
does not grow too rapidly, or decay to slowly, relative to the
exponential rate of growth or decay of the resolvent). It is known
from~\cite{jamrcs10} in the case of a particular scalar functional
differential equation with finite delay, for which the solution of
the characteristic equation with largest real part is real and
simple, and for which the noise intensity is constant, that the
solution of the stochastic equation inherits exactly the rate of
growth of the resolvent. It is natural to ask whether a result of
this kind can be generalised to deal with finite dimensional
equations, of both finite delay and Volterra type, for which there
may be many solutions of the characteristic equation which have the
same real part, need not be simple, nor even be real solutions.

It is a longstanding theme in the asymptotic theory of differential
equations, and especially of linear equations, to ask the question:
how large can a forcing or perturbation term be, so that the
perturbed differential system preserves the asymptotic behaviour of
the underlying unperturbed equation. Investigations of this type
were systematically initiated by Hartman and Wintner in the
1950's~\cite{hartwint:1953,hartwint:1954,hartwint:1954b,hartwint:1955}.
More recently, there have been many interesting contributions
concerning the asymptotic behaviour of functional differential
equations: the literature is quite large, but some important and
representative papers include Cruz and Hale~\cite{cruzhale:1971},
Haddock and Sacker~\cite{hadsack:1980}, Arino and Gy\H{o}ri
\cite{arinogyori:1989}, Castillo and
Pinto~\cite{castillopinto:2002}, Gy\H{o}ri and
Pituk~\cite{gyoripituk:1995}, Pituk~\cite{pituk:1999, pituk:2006},
and Gy\H{o}ri and Hartung~\cite{gyorihartung:2010} among many
others. Already, some results for stochastic Volterra equations with
state--independent perturbations suggest that results of this type
may also be available in the random case
Appleby~\cite{app:2004gyori}.

It is one of the goals of this paper to demonstrate that very sharp
conditions can be identified on the intensity of the perturbations
under which the asymptotic behaviour of the deterministic equations
is preserved. Moreover, we show that the results apply to a wide
class of affine stochastic functional differential equation, and
examples and underlying admissibility results show that there is the
potential for our work to apply to a wider class yet.

Our results for the solution $X$ of functional differential
equations have the form
\begin{equation} \label{eq.intro3}
\lim_{t\to\infty} \left\{\frac{X(t,\omega)}{\gamma(t)} -
S(t,\omega)\right\}=0, \quad \text{a.s. and in mean square}
\end{equation}
where $\gamma:(0,\infty)\to (0,\infty)$ is a deterministic real
exponential polynomial, and $S$ is a random sinusoidal vector, whose
``frequencies'' are deterministic but whose ``amplitudes'' or
``multipliers'' are multidimensional normal random variables which
are path--dependent (in the case where the zeros of the
characteristic equation with largest real part are real, $S$ is a
constant random vector). These ``multipliers'' turn out to be
identifiable linear functionals of the Brownian motion, the noise
intensity $\Sigma$, and of the initial function or condition,
because we have an explicit formula for these multipliers in terms
of the solutions of the characteristic equation with largest real
part. Similar multipliers emerge in papers of Appleby, Devin and
Reynolds on stochastic Volterra equations whose solutions have
Gaussian limits~ \cite{jasddr:2006, jasddr:2007}. Moreover, the
joint distribution of these random limits is known exactly, because
the mean and covariance matrix of the Gaussian limit can be computed
explicitly in terms of the components of the random vector. This has
already proved of interest in~\cite{jamrcs10} where the form of the
multiplier can be used to describe the mechanism by which financial
market bubbles can start. Our results here are also superior to
those in Appleby and Daniels~\cite{App_Dan} (i.e. Chapter 5) in
which a limit formula for asset returns of the form
\eqref{eq.intro3} is found for a nonautonomous stochastic functional
differential equation. The method of asymptotic analysis, which
applies the \emph{deterministic} admissibility theory pathwise,
shows that the distribution of $S$ is Gaussian, but does not enable
a formula for the variance to be determined. These examples from
finance demonstrate the utility of an authentically stochastic
admissibility theory in finding the exact form of the limiting
multiplier.

\section{Mathematical Preliminaries}
\subsection{Notation and terminology}
Let $\mathbb{Z}$ be the set of integers,
$\mathbb{Z}^{+}=\{n\in\mathbb{Z}:n\geq0\}$ and $\R$ the set of real
numbers. We denote by $\Rp$ the half-line $[0,\infty)$. The complex
plane is denoted by $\C$ and $\C_0:=\{z\in\C:\, \Re (z)\ge 0\}$,
where $\Re(z)$ and $\Im(z)$ denote the real and imaginary parts of
any complex number $z$.
 If $d$ is a positive integer, $\R^d$ is the space of
$d$-dimensional column vectors with real components and
$\R^{d_1\times d_2}$ is the space of all $d_1 \times d_2$ real
matrices. The identity matrix on $\R^{d\times d}$ is denoted by
$\I_d$, while $0_{d_1,d_2}$ represents the matrix of zeros in
$\mathbb{R}^{d_1\times d_2}$. Let $A\in\R^{d\times d}$ then det($A$)
denotes the determinant of the square matrix $A$. $A^T$ denotes the
transpose of any $A\in\R^{d_1\times d_2}$. The absolute value of $A
=A_{i,j}$ in $\R^{d_1\times d_2}$ is the matrix given by
$(|A|)_{i,j}=|A_{i,j}|$.

We employ the standard Landau notation: if
$f:\mathbb{C}\to\mathbb{C}$ and $g:\mathbb{C}\to\mathbb{R}$, we
write $f=O(g)$ as $|z|\to\infty$ if there exist $z_0>0$ and $M>0$
such that $|f(z)|\leq M|g(z)|$ for all $|z|>z_0$, for a matrix
valued function the Landau notation is applied element-wise. For any
two functions $U:\Rp\to\R^{d_1\times d_2}$ and
$V:\Rp\to\R^{d_2\times d_3}$. we define the \emph{convolution} of
$\{(U\ast V)(t)\}_{t\geq0}$ by
\[
    (U\ast V)(t)=\int_{0}^t U(t-s)V(s)\,ds, \quad t\geq 0.
\]
In this paper the \emph{Laplace transform} of a sequence $U$ in
$\R^{d_1\times d_2}$ is the function defined by
\[
    \tilde{U}( \lambda)= \int_{0}^\infty \e^{-\lambda s}U(s)\,ds,
\]
provided $\lambda$ is a complex number for which the integral
converges absolutely. A similar definition pertains for the Laplace
transform of a measure, \cite[Definitions~2.1,~2.2]{GrLoSt90} and
for functions with values in other spaces.

Let $BC(\Rp;\R^{d_1\times d_2})$ denote the space matrices whose
elements are  bounded continuous functions. The abbreviation
\text{\it a.e.} stands for \text{\it almost everywhere}. The space
of continuous and continuously differentiable functions on $\Rp$
with values in $\R^{d_1\times d_2}$ is denoted by
$C(\Rp;\R^{d_1\times d_2})$ and $C^1(\Rp;\R^{d_1\times d_2})$
respectively. While $C^{1,0}(\Delta;\R^{d_1\times d_2})$ represents
the space of functions which are continuously differentiable in
their first argument and continuous in their second argument, over
some two--dimensional space $\Delta$. For any scalar function
$\varphi$, the space of weighted $p^{th}$integrable functions is
denoted by
\[
    L^p(\Rp;\R^{d_1\times d_2};\varphi) :=
    \{ f:\Rp\to\R^{d_1\times d_2}:\int_{0}^{\infty}\varphi(s)|f(s)_{i,j}|^p \,ds <+\infty, \text{ for all } i,j \},
\]
when $\varphi=1$, we do not include it in our notation, i.e.
$L^p(\Rp;\R^{d_1\times d_2};1)=L^p(\Rp;\R^{d_1\times d_2})$.

For any vector $x\in\R^d$ the norm $\norm{\cdot}$ denotes the
Euclidean norm,
     $\norm{x}^2=\sum_{j=1}^{d}x_j^2$
and the infinity norm, $|\cdot|_{\infty}$, is defined by
$|x|_{\infty} = \max_{i=1,...,d}\left(|x_1|,...,|x_d| \right)$.

While for a matrix norm we use the Frobenius norm, for any
$A=(a_{i,k})\in\R^{n\times d}$
\[
    \norm{A}_{F}^2 = \sum_{i=1}^{n}\sum_{k=1}^{d}|a_{i,k}|^2.
\]
As both $\R^d$ and $\R^{d\times d}$ are finite dimensional Banach
spaces all norms are equivalent in the sense that for any other
norm,  $\norm{\cdot}$, one can find universal constants
$d_1(n,d)\leq d_2(n,d)$ such that
\[
    d_1\norm{A}_F \leq \norm{A} \leq d_2\norm{A}_F.
\]
Thus there is no loss of generality in using the Euclidean and
Frobenius norms, which for ease of calculation, are used throughout
the proofs of this paper. Moreover we remark that the Frobenius norm
is a {\it consistent matrix norm}, i.e. for any $A\in\R^{n_1\times
n_2}, B\in\R^{n_2\times n_3}$
\[
    \norm{AB}_F \leq \norm{A}_F \norm{B}_F.
\]

For any matrix $C\in\R^{n\times d}$ we say $C\geq0$ if
$(C)_{i,j}\geq0$ for all $i,j$. Also, we say for any matrices
$A,B\in\R^{n\times d}$ that $A\leq B$ if $B-A\geq0$. We will use the
fact that $\norm{A}\leq\norm{B}$ whenever $0\leq A\leq B$.

We also will require some notation and results regarding finite
measures on sub--intervals of the real line. Let $M(J,\R^{d\times
d^\prime})$ be the space of finite Borel measures on $J$ with values
in $\R^{d\times d^\prime}$, where $J$ shall be either $\Rp$ or
$[-\tau,0]$. The total variation of a measure $\nu$ in
$M(J,\R^{d\times d^\prime})$ on a Borel set $B\subseteq J$ is
defined by
\begin{align*}
 \abs{\nu}\!(B):=\sup\sum_{i=1}^N \abs{\nu(E_i)},
\end{align*}
where $(E_i)_{i=1}^N$ is a partition of $B$ and the supremum is
taken over all partitions. The total variation defines a positive
scalar measure $\abs{\nu}$ in $M(J,\R)$. If one specifies
temporarily the norm $\abs{\cdot}$ as the $l^1$-norm on the space of
real-valued sequences and identifies $\R^{d\times d^\prime}$ by
$\R^{dd^\prime}$ one can easily establish for the measure
$\nu=(\nu_{i,j})_{i,j=1}^{d,d'}$ the inequality
\begin{align}\label{eq.totalvarest}
 \abs{\nu}\!(B)\le  C \sum_{i=1}^d\sum_{j=1}^{d'} \abs{\nu_{i,j}}\!(B)
 \qquad\text{for every Borel set }B\subseteq \Rp
\end{align}
with $C=1$. Then, by the equivalence of every norm on
finite-dimensional spaces, the inequality \eqref{eq.totalvarest}
holds true for the arbitrary norms $\abs{\cdot}$ and some constant
$C>0$. Moreover, as in the scalar case we have the fundamental
estimate
\begin{align*}
 \abs{\int_{J} \nu(ds)\, f(s)} \le \int_{J}\abs{\nu}\!(ds) \,\abs{f(s)}
\end{align*}
for every function $f:J\to\R^{d^\prime \times d^{\prime\prime}}$
which is $\abs{\nu}$-integrable.

\begin{definition}
A positive function $\varphi$ defined on $\R$ is called
submultiplicative, if $\varphi(0)=1$, and
\[
    \varphi(s+t)\leq\varphi(s)\varphi(t),
\]
for all $s,t\in\R.$
\end{definition}
We also define the limits
\[
    \alpha_{\varphi}:= -\lim_{t\to-\infty}\frac{\ln(\varphi(t))}{t},
    \quad \omega_{\varphi}:= -\lim_{t\to\infty}\frac{\ln(\varphi(t))}{t}.
\]
Which always exist when $\varphi$ is a submultiplicative function,
c.f. \cite[Lemma~4.1]{GrLoSt90}.

We define the following modes of convergence:
\begin{definition}
The $\mathbb{R}^{n}$-valued stochastic process $\{X(t)\}_{t\geq0}$
converges in mean-square to $X_{\infty}$ if
\[
    \lim_{t\to\infty}\mathbb{E}[\norm{X(t)-X_{\infty}}^2] = 0.
\]
\end{definition}
\begin{definition}
If there exists a $\mathbb{P}$--null set $\Omega_0$ such that for
every $\omega\not\in\Omega_0$ the following holds
\[
    \lim_{t\to\infty}X(t,\omega) = X_{\infty}(\omega),
\]
then we say $X$ converges almost surely (a.s.) to $X_{\infty}$.
\end{definition}

\subsection{Admissibility theory for linear stochastic Volterra operators}\label{sect:stochlim}
The main results in this paper are established using convergence
results proven in \cite{AppDanRey} for linear stochastic Volterra
operators. Since these results are used extensively throughout, they
are stated here for the convenience of the reader. A important
corollary of these results, which is of especial use in our
asymptotic analysis of affine stochastic equations, is given in the
next section.

We consider the following hypotheses: let $\Delta\subset
\mathbb{R}^2$ be defined by
\[
\Delta=\{(t,s): 0\leq s\leq t<+\infty\}
\]
and
\begin{equation}\label{eq.Hcns2}
    H:\Delta\to \R^{n\times n} \text{ is continuous.}
\end{equation}
We first characterise, for $f\in C([0,\infty);\mathbb{R}^{n\times
d})$ with bounded norm, the convergence of the stochastic process
$X_f=\{X_f(t):t\geq 0\}$ defined by
\[
X_f(t)=\int_0^t H(t,s)f(s)\,dB(s), \quad t\geq 0
\]
to a limit as $t\to\infty$ in \textit{mean--square}, where
$B(t)=\{B_1(t),B_2(t),...,B_d(t)\}$ is a vector of mutually
independent standard Brownian motions. For the definition of a
stochastic integral in higher dimensions and the result
corresponding to It\^o's isometry we refer the reader to
\cite[Definition~1.5.20 and Theorem~1.5.21]{Mao2}.

Before discussing the convergence in mean square, we note that \eqref{eq.Hcns2} is
sufficient to guarantee that $X_f(t)$ is a well--defined random
variable for each fixed $t$. Therefore the family of random
variables $\{X_f(t):t\geq 0\}$ is well--defined, and $X_f$ is indeed
a process, and for each fixed $t$ the random variable $X_f(t)$ is
$\mathcal{F}^{B}(t)$-adapted. Condition \eqref{eq.Hcns2} also
guarantees that $\mathbb{E}[X_f(t)^2]<+\infty$ for each $t\geq 0$.
Since $f\mapsto X_f$ is linear, and the family $(X_f(t))_{t\geq 0}$
is Gaussian for each fixed $f$, the limit should also be Gaussian
and linear in $f$, as well as being an
$\mathcal{F}^B(\infty)$--measurable random variable. Therefore, a
reasonably general form of the limit should be
\[
X_f^\ast:=\int_0^\infty H_\infty(s) f(s)\,dB(s),
\]
where we would expect $H_\infty$ to be a function independent of
$f$.  Our first main result, which is proven in~\cite{{AppDanRey}},
characterises the conditions under which $X_f(t)\to X_f^\ast$ in
mean square as $t\to\infty$ for each $f$.
\begin{theorem} \label{thm.msqcharacterise2}
Suppose that $H$ obeys \eqref{eq.Hcns2}. Then the statements
\begin{itemize}
\item[(A)] There exists
$H_\infty\in C([0,\infty);\mathbb{R}^{n\times n})$ such that
$\int_{0}^{\infty}\norm{H_{\infty}(s)}^2 ds <+\infty$  and
\begin{equation}\label{eq.HtoHinfty2}
    \lim_{t\to\infty} \int_{0}^t \norm{H(t,s)-H_\infty(s)}^2\,ds=0.
\end{equation}
\item[(B)] There exists $H_\infty \in C([0,\infty);\mathbb{R}^{n\times n})$ such that
for each $f\in BC(\Rp;\mathbb{R}^{n\times d})$,
\begin{equation}\label{eq.meansquareconv2}
    \lim_{t\to\infty} \mathbb{E}\left[\norm{\int_0^t H(t,s)f(s)\,dB(s) - \int_0^\infty H_\infty(s)f(s)\,dB(s) }^2\right]=0
\end{equation}
\end{itemize}
are equivalent.
\end{theorem}
We now consider the almost sure convergence of $X_f(t)$ as
$t\to\infty$ to a limit. Our next main result states that if we have
convergence in an a.s. sense, we must also have convergence in a
mean square sense.
\begin{theorem} \label{thm.ascharacterise2}
Suppose that $H$ obeys \eqref{eq.Hcns2} and there exists $H_\infty
\in C([0,\infty);\mathbb{R}^{n\times n})$ such that for each $f\in
BC([0,\infty);\mathbb{R}^{n\times d})$,
\begin{equation}\label{eq.asconv2}
\lim_{t\to\infty} \int_0^t H(t,s)f(s)\,dB(s) = \int_0^\infty
H_\infty(s)f(s)\,dB(s), \quad \text{a.s.}
\end{equation}
Then \eqref{eq.HtoHinfty2} and \eqref{eq.meansquareconv2} hold.
\end{theorem}
\eqref{eq.HtoHinfty2} is a necessary condition for a.s. convergence.
It is of course natural to then ask whether \eqref{eq.HtoHinfty2} is
sufficient. By means of examples, it is shown in \cite{AppDanRey}
that in general additional conditions are needed in order for
\eqref{eq.asconv2} to hold. We now state our main result which
guarantees a.s. convergence of the stochastic integral.
\begin{theorem} \label{thm.assuff2}
Suppose that $H$ obeys \eqref{eq.Hcns2} and also that $H\in
C^{1,0}(\Delta;\mathbb{R}^{n\times n})$. Suppose also that there
exists $H_\infty\in C([0,\infty);\mathbb{R}^{n\times n})$ such that
$\int_{0}^{\infty}\norm{H_{\infty}(s)}^2 ds <+\infty$ and
\begin{equation} \label{eq.Htilderateto02}
\lim_{t\to\infty} \int_0^t \norm{H(t,s)-H_\infty(s)}^2\,ds \cdot
\log t=0,
\end{equation}
and
\begin{multline}
\label{eq.H1to02}
\text{There exists $q\geq 0$ and $c_q>0$ such that } \\
\int_0^t \norm{H_1(t,s)}^2\,ds \leq c_q(1+t)^{2q}, \quad
\norm{H(t,t)}^2\leq c_q(1+t)^{2q}.
\end{multline}
Then $H$ obeys \eqref{eq.asconv2}.
\end{theorem}

\begin{remark} \label{rem.strengthenedtildeHto0}
We notice that \eqref{eq.Htilderateto02} implies a given rate of
decay to zero of $\int_0^t (H(t,s)-H_\infty(s))^2\,ds$ as
$t\to\infty$. This strengthens the hypothesis \eqref{eq.HtoHinfty2}
which is known, by Theorem~\ref{thm.ascharacterise2}, to be
necessary.
\end{remark}

\begin{remark}\label{rk:Hint2}
    While the pointwise bound on $H(t,t)$ given by
     $\norm{H(t,t)}^2 \leq c_q(1+t)^{2q}$ in Theorem~\ref{thm.assuff2}
 may appear quite mild,
    one may prefer an integral condition to this pointwise bound as this would allow for $H(t,t)$ to potentially have
    ``thin spikes" of larger than polynomial order. In \cite{AppDanRey} it is pointed out that this pointwise condition  can be replaced by
    \begin{equation}\label{eq:Hint2}
        \lim_{k\to\infty}\int_{k^{\theta}}^{(k+1)^{\theta}}\norm{H(s,s)}^2 \,ds \cdot \log k = 0, \quad \text{ for }0<\theta<1/(1+2q),
    \end{equation}
    where the limit is taken through the integers. Nevertheless for simplicity we retain the condition on $H(t,t)$ in the statement
    of Theorem~\ref{thm.assuff2}.
\end{remark}

\subsection{Asymptotic behaviour of a stochastic convolution integral} \label{sect:asystochdetconv}
In this section we state a key theorem which will be used to
determine the asymptotic behaviour of solutions of \text{\it
Volterra linear SFDEs} and \text{\it linear SFDEs with finite delay}
with state--independent noise intensity. This theorem is a
consequence of the stochastic admissibility results stated in
Section~\ref{sect:stochlim}.

To see the connection between these admissibility results and the
asymptotic behaviour of such affine equations, we note that both
classes of equations can be written in the form
\[
    dX(t) = \bigl(f(t)+L(X_t)\bigr)dt + \Sigma(t) \, dB(t), \quad t\geq0,
\]
where $L$ is a linear functional, $\Sigma\in C(\Rp;\R^{d\times
d'})$, $f\in C(\Rp;\R^{d})$, $B$ is a standard $d'$-dimensional
Brownian vector and the solution $X$ lies in $\mathbb{R}^{d}$. For
any $y:\R\to\R^{d\times n}$ we define the \emph{segment}
$y_t:\R\to\R^{d\times n}:s\mapsto y(t+s)$ for any
$n,d\in\mathbb{Z}^+$. An appropriate initial condition is also
imposed. The associated deterministic equation is
\[
    x'(t)=L(x_t), \quad t\geq0,
\]
with the same initial value as the stochastic equation. Also
defining the differential resolvent, $r$,
\begin{equation}\label{eq:linr}
    r'(t)=L(r_t), \quad t\geq0, \qquad r(0)=I_d,
\end{equation}
allows one to write the variation of parameters formula, for
$t\geq0$,
\[
    X(t)=x(t)+ \int_{0}^{t}r(t-s)f(s)\,ds + \int_{0}^{t}r(t-s)\Sigma(s)\, dB(s).
\]
The asymptotic behaviour of $x$ and $r$ is primarily known from the
theory of deterministic linear differential equations and so one may
now apply the admissibility theory of Section~\ref{sect:stochlim} to
determine the asymptotic behaviour of the stochastic convolution
integral, $\int_{0}^{t}r(t-s)\Sigma(s)\, dB(s)$, and hence of $X$,
providing that the diffusion, $\Sigma$, does not grow too rapidly.

\begin{proposition}\label{prop:Y}
    Let $\alpha\in\R$, $N$ be some finite positive integer, $\{\beta_j\}_{j=1}^{N}$ be a sequence of some
    real constants and $(P_j)_{j=1}^{N}$ and $(Q_j)_{j=1}^N$ be sequences of $d\times d$ matrix--polynomials of degree $n$,
    for some positive integer $n$, and in particular
\[
    P_j(t)=t^n P_j^* + O(t^{n-1}), \quad Q_j(t)=t^n Q_j^* + O(t^{n-1}).
\]
where at least one of $P_j^*,Q_j^*\neq 0$ for all $j\in
\{1,\ldots,N\}$. Suppose $R$ is a.e. absolutely continuous
         and is defined such that
         it obeys, for some $\epsilon>0$, the asymptotic estimates
        \begin{align}\label{eq:Rorder}
            R(t) &=
            \begin{cases}
                O(\e^{(\alpha-\epsilon) t}),    &\text{ if } n=0 \\
                O(\e^{\alpha t}t^{n-1}),        &\text{ if } n\geq1
            \end{cases}, \quad \text{ as }  t\to\infty, \\
            \label{eq:R'order}
            R'(t) &=
            \begin{cases}
                    O(\e^{(\alpha-\epsilon) t}),    &\text{ if } n=0 \\
                                O(\e^{\alpha t}t^n), &\text{ if } n\geq1
             \end{cases},
             \quad \text{ as }  t\to\infty.
        \end{align}
  and suppose that $r$ is given by
    \begin{equation}\label{eq:rR1}
        r(t) = \sum_{j=1}^{N}\e^{\alpha t}\{P_j(t)\cos(\beta_j t) + Q_j(t)\sin(\beta_j t)\} + R(t), \quad t\geq0.
    \end{equation}
Let $\Sigma\in C([0,\infty);\mathbb{R}^{d\times d'})$ be continuous
with
\begin{equation}\label{eq:SigL2}
\int_0^\infty e^{-2\alpha t}\|\Sigma(t)\|^2\,dt < + \infty.
\end{equation}
        Let $Y$ be the process defined by
        \begin{equation}\label{eq:Yvar}
            Y(t)=\int_{0}^{t}r(t-s)\Sigma(s) \, dB(s), \quad t\geq0, \quad Y(0)=0.
        \end{equation}
        Then
    \begin{equation}\label{eq:asyY}
        \lim_{t\to\infty} \left( \frac{Y(t)}{t^n\e^{\alpha t}}
         -\sum_{j=1}^{N}\{L_{1,j}\sin(\beta_j t) + L_{2,j}\cos(\beta_j t)\} \right) =0, \quad \text{a.s.}
    \end{equation}
    where
    \begin{subequations} \label{def.L}
    \begin{align} \label{def.L1}
        L_{1,j} &:= \int_{0}^{\infty}\e^{-\alpha s}\{ P_j^*\sin(\beta_j s) + Q_j^*\cos(\beta_j s) \}\Sigma(s) \,dB(s), \\
        \label{def.L2}
        L_{2,j} &:= \int_{0}^{\infty}\e^{-\alpha s}\{ P_j^*\cos(\beta_j s) - Q_j^*\sin(\beta_j s) \}\Sigma(s) \,dB(s).
    \end{align}
    \end{subequations}
\end{proposition}
The square integrability, $L^2(0,\infty)$, of the noise term, i.e.
\eqref{eq:SigL2}, is a usual condition to have when dealing with
stochastic terms. When ascertaining asymptotic behaviour of
deterministic forcing functions it is more typical to require an
absolute integrability condition, $L^1(0,\infty)$. This is indeed
what is required in Corollary~\ref{cor:Yf}, i.e. \eqref{eq:fL1}.
Proposition~\ref{prop:Y} is shown to be robust with respect to
deterministic perturbations.
\begin{corollary}\label{cor:Yf}
Let $\alpha\in\mathbb{R}$, $N\in\mathbb{Z}^+/\{0\}$. Let
$\{\beta_j\}_{j=1}^{N},\{P_j\}_{j=1}^{N},\{Q_j\}_{j=1}^{N},r,R,\Sigma$
and $Y$ be as defined in Proposition~\ref{prop:Y}, with
\eqref{eq:SigL2} holding. Let $f\in C([0,\infty),\mathbb{R}^{d})$
with
\begin{equation}\label{eq:fL1}
    \int_{0}^{\infty}\e^{-\alpha t}|f(t)|dt <+\infty.
\end{equation}
 Let $V$ be the process defined by
\begin{equation}\label{eq:VpY}
    V(t) = \int_{0}^{t}r(t-s)f(s)ds + Y(t), \quad t\geq0, \quad V(0)=0.
\end{equation}
Then
\[
    \lim_{t\to\infty}\left(\frac{V(t)}{t^n\e^{\alpha t}}
    -\sum_{j=1}^{N}\{M_{1,j}\sin(\beta_j t) + M_{2,j}\cos(\beta_j t)\} \right) =0, \quad a.s.
\]
where
\begin{align*}
    M_{1,j} &= L_{1,j} + \int_{0}^{\infty}\e^{-\alpha s}\{P^*_j \sin(\beta_j s) + Q_j^*\cos(\beta_j s)\}f(s) \,ds, \\
    M_{2,j} &= L_{2,j} + \int_{0}^{\infty}\e^{-\alpha s}\{P_j^* \cos(\beta_j s) - Q_j^*\sin(\beta_j s)\}f(s) \,ds
\end{align*}
and where $L_{1,j}$ and $L_{2,j}$ are given by
Proposition~\ref{prop:Y}.
\end{corollary}

\section{Affine Stochastic Functional Differential Equations}
The organisation of this section is as follows: in the first part of
this section, we discuss the structure of solutions of affine
stochastic Volterra functional, and show that the resolvent of the
underlying deterministic equation can play the role of the function
$r$ introduced in the statement of Proposition~\ref{prop:Y}, modulo
some deterministic asymptotic estimates. The second part of the
section contains a parallel discussion for the solution of the
affine stochastic functional differential equation with finite
delay. These preliminary discussions pave the way for main
asymptotic results for both stochastic Volterra and finite delay
equations which are stated in Section~\ref{sect:actualresults}.

\subsection{Volterra linear functional equations}\label{sub.Vol}
A Shea-Wainger theorem is developed  in \cite{apprie} which relates
the location of the roots of a characteristic equation to the
solution of a Volterra linear SFDE lying in a weighted
$L^{p}$-space. We reproduce the set-up of those equations here.

Let $(\Omega, \mathcal{F}, \P)$ be a complete probability space
equipped with a filtration $(\mathcal{F}_{t})_{t \geq 0}$, and let
$(B(t))_{t \geq 0}$ be a standard $d^\prime$-dimensional Brownian
motion on this probability space. Consider the stochastic
integro-differential equation with stochastic perturbations of the
form
\begin{align}\label{eq.stoch}
\begin{split}
 dX(t)&= \left(f(t) + \int_{[0,t]} \mu(ds)\, X(t-s)\right)dt + \Sigma(t) \,dB(t)
 \quad\text{for }t\ge 0,\\
  X(0)&=X_0,
\end{split}
\end{align}
where $\mu$ is a measure in $M(\Rp,\R^{d\times d})$, $\Sigma\in
C(\Rp;\R^{d\times d'})$, $f\in C(\Rp;\R^{d})$. The initial condition
$X_0$ is an $\R^d$-valued, $\mathcal{F}_0$-measurable random
variable with $\mathbb{E}\abs{X_0}^2<\infty$. The existence and
uniqueness of a continuous solution $X$ of \eqref{eq.stoch} with
$X(0)=X_0$ $\P$-a.s. is covered in Berger and
Mizel~\cite{BergMiz:2}, for instance. Independently, the existence
and uniqueness of solutions of stochastic functional equations was
established in It\^o and Nisio~\cite{ItoNisio:64} and
Mohammed~\cite{Moh:84}.

The so-called {\em fundamental solution or resolvent of
\eqref{eq.stoch}} is the matrix-valued function $r:\Rp\to\R^{d\times
d}$, which is the unique solution of
\begin{align}    \label{eq.fund1}
r'(t) = \int_{[0,t]}  \mu(ds)\, r(t-s) \quad\text{for } t
\geq0,\quad  r(0)=\I_d.
\end{align}
In the following Proposition, we give a variation of constants
formula for the solution of \eqref{eq.stoch} in terms of the
solution $r$ of \eqref{eq.fund1}. The proof is a simple adaptation
of a result of Rei\ss, Riedle and van Gaans \cite[Lemma 6.1]{rrg}.
\begin{proposition} \label{prop.stochresrep}
Let $\mu\in M(\Rp,\R^{d\times d})$, $\Sigma\in C(\Rp;\R^{d\times
d'})$, $f\in C(\Rp;\R^{d})$, and suppose that $r$ is the unique
continuous solution of \eqref{eq.fund1}. Then the unique continuous
adapted process $X$ which obeys \eqref{eq.stoch} is given by
\begin{align}\label{eq.stochconfor}
 X(t)=r(t)X_0 + \int_{0}^{t}r(t-s)f(s)\,ds + \int_0^t r(t-s)\Sigma(s)\,dB(s)\quad\text{$\P$-a.s.}
\end{align}
\end{proposition}
The proof is given in Section~\ref{sect:proof3}.

The chief difficulty in estimating the asymptotic behaviour of $X$
therefore lies in determining the asymptotic behaviour of the
stochastic convolution integral on the right--hand side of
\eqref{eq.stochconfor}. However, we argue below that the solution
$r$ of \eqref{eq.fund1} can be decomposed as in \eqref{eq:rR1}, with
leading order exponential polynomial behaviour and the remainder
terms obeying the growth estimates of the form \eqref{eq:Rorder} and
\eqref{eq:R'order}. Then, the last term on the righthand side of
\eqref{eq.stochconfor} is of the form of the process $Y$ defined in
\eqref{eq:Yvar}, and therefore, under appropriate growth conditions
on $\Sigma$, Proposition~\ref{prop:Y} can be applied to this term.

In order to do this, we start by defining the real number
$\alpha^\ast$ by
\begin{equation}\label{eq:tranE}
    \alpha^* = \inf\{a\in\R: \int_{[0,\infty)}\e^{-a s}|\mu|(ds) \text{ is well--defined and finite} \}.
\end{equation}
Then the function $h_\mu:\mathbb{C}\to\C$ defined by
\[
    h_\mu(\lambda) = \text{det}\left( \lambda I_d - \int_{[0,\infty)} \e^{-\lambda s}\mu(ds) \right).
\]
is well--defined for $\Re(\lambda)>\alpha^*$.

Define also the set
\[
    \Lambda=\{\lambda\in\C: h_\mu(\lambda)=0 \}.
\]
The function $h_\mu$ is analytic, and so the elements of $\Lambda$
are isolated. Define
\begin{equation}\label{eq.alpha1}
    \alpha := \sup\{ \Re(\lambda) : h_\mu(\lambda)=0 \}.
\end{equation}
It is always the case that such an $\alpha$ is finite, we assume
however that $\alpha^*<\alpha$. Because the solution $r$ obeys an
exponentially growing or decaying upper bound, this is equivalent to
assuming that there exists $\lambda\in\mathbb{C}$ with
$\Re(\lambda)>\alpha^*$ for which $h_\mu(\lambda)=0$.

With the assumption $\alpha^\ast<\alpha$, there exists
$\delta\in(0,\alpha-\alpha^\ast)$. By the Riemann--Lebesgue lemma,
cf. e.g. \cite[Thm.~2.2.7~(i)]{GrLoSt90},
for such a $\delta>0$ there exists $M=M(\delta)>0$ such that
$h_\mu(\lambda)\neq 0$ for all $\lambda\in\mathbb{C}$ such that
$\alpha^*<\alpha-\delta\leq \Re(\lambda)\leq \alpha+\delta$ and
$|\Im(\lambda)|\geq M(\delta)$. If $K=\{\lambda\in
\mathbb{C}:0<|\Re(\lambda)-\alpha|<\delta, \,|\Im(\lambda)|\leq
M(\delta)\}$, the fact that $h_\mu$ is analytic ensures that there
are at most finitely many zeros of $h_\mu$ in $K$. Therefore, there
exists a minimal $\varepsilon\in (0,\delta]$ such that
$h_\mu(\lambda)\neq 0$ for all
$\alpha-\varepsilon\leq\Re(\lambda)<\alpha$, and therefore there
exists $\delta'=\alpha-\varepsilon$ such that $h_\mu(z)\neq 0$ for
all $\Re(z)=\delta'$. Define $\varphi(t)=e^{-\delta't}$ for $t\in
\mathbb{R}$. Then $\varphi$ is a submultiplicative weight function
on $\mathbb{R}$ for which
$\omega_\varphi=\alpha_\varphi=\delta'=\alpha-\varepsilon$. Define
$\Lambda_\varepsilon=\{\lambda\in
\Lambda:\Re(\lambda)>\alpha-\varepsilon\}$. Clearly
$\Lambda_\varepsilon$ is a set with only finitely many elements, as
is $\Lambda'=\{\lambda\in \Lambda:\Re(\lambda)=\alpha\}$. Then by
Theorem~7.2.1 in \cite{GrLoSt90}, there exists an a.e. absolutely
continuous function $q$ such that $q, q'\in
L^1(\Rp;\varphi;\mathbb{R}^{d\times d})$ and
\begin{equation}  \label{eq.solsumrepvol}
r(t)=\sum_{\lambda_j\in \Lambda_\varepsilon, \Im(\lambda_j)\geq0}
e^{\alpha_j t} \{P_j(t)\cos(\beta_j t)+Q_j(t)\sin(\beta_j t)\} +
q(t), \quad t\geq 0.
\end{equation}
where $\Re(\lambda_j)=\alpha_j$ and $\Im(\lambda_j)=\beta_j$, and
where $P_j$ and $Q_j$ are matrix--valued polynomials of degree
$n_j$, with $n_j +1$ being the order of the pole
$\lambda_j=\alpha_j+i\beta_j$ of $[h_\mu]^{-1}$. We remark that
$n_j$ (the ascent of $\lambda_j$) is less than or equal to the
multiplicity of the zero $\lambda_j=\alpha_j+i\beta_j$ of $h_\mu$.

Let $n$ denote the highest degree of all polynomials associated with
roots in $\Lambda'$ and
let $\lambda_1,...,\lambda_N$ be the finitely many roots in
$\Lambda'$
which have associated polynomials of this degree and have
$\Im(\lambda_j)=\beta_j\geq0$. We associate with each such
$\lambda_j=\alpha+i\beta_j$ the  matrix polynomials $P_j$ and
$Q_{j}$ of degree $n$ in \eqref{eq.solsumrepvol}. Therefore we may
write
\begin{equation}  \label{eq.PjQjstarvol}
    P_j(t)=t^n P_j^* + O(t^{n-1}), \quad Q_j(t)=t^n Q_j^* + O(t^{n-1}).
\end{equation}
where at least one of $P_j^*$ and $Q_j^*$ are not equal to the zero
matrix, for each $j\in\{1,...,N\}$. The precise values of $P_j^*$
and $Q_j^*$ can be determined from the Laurent series of the inverse
of the characteristic function, $h_{\mu}$, expanded about
$\lambda_j$, i.e.
\begin{equation}\label{eq:tayK}
    \left[\lambda I_d - \int_{[0,\infty)}\e^{-\lambda s}\mu(ds)\right]^{-1}
    = \sum_{m=0}^{n}\frac{m! \,K_{j,m}}{(\lambda-\lambda_j)^{m+1}} + \hat{q}_j(\lambda),
\end{equation}
where the remainder term $\hat{q}_j(\lambda)$ is analytic at
$\lambda_j$. If $\lambda_j$ is real then $P_j^*=K_{j,n}$, otherwise
$P^*_j:=2\,\Re(K_{j,n})$ and  $Q^*_j:=-2\,\Im(K_{j,n})$. We note
that \eqref{eq:tayK} defines the value of $n$.

Now define
\begin{equation}\label{eq:rRvol}
   R(t) =  r(t) - \sum_{j=1}^{N}\e^{\alpha t}\{P_j(t)\cos(\beta_j t) + Q_j(t)\sin(\beta_j t)\}, \quad t\geq0.
\end{equation}
It is clear that $R$ is a.e. absolutely continuous. Therefore, by
virtue of the decomposition in \eqref{eq:rR1} in the statement of
Proposition~\ref{prop:Y}, if the growth estimates \eqref{eq:Rorder}
and \eqref{eq:R'order} can be established for $R$ defined by
\eqref{eq:rRvol}, we will be in a excellent position to apply
Proposition~\ref{prop:Y} to the stochastic convolution term on the
right--hand side of \eqref{eq.stochconfor}. The relevant estimates
will be provided in Lemma~\ref{lm:Rvol}, which is stated in
Section~\ref{sect:actualresults}.
%
%
%
\subsection{Finite delay linear functional equations.}\label{sub.finite}
The exact rate of growth of the running maxima of solutions of
affine SFDEs with finite memory is discussed in \cite{appmaowu}. We
reproduce the set-up of those equations here.

Let $(\Omega, \mathcal{F}, \P)$ be a complete probability space
equipped with a filtration $(\mathcal{F}_{t})_{t \geq 0}$, and let
$(B(t))_{t \geq 0}$ be a standard $d^\prime$-dimensional Brownian
motion on this probability space. Consider the stochastic
integro-differential equation of the form
\begin{align}\label{eq.stoch2}
\begin{split}
 dX(t)&= \left(f(t) + \int_{[-\tau,0]} \nu(ds)\, X(t+s)\right)dt + \Sigma(t)\,dB(t)
 \quad\text{for }t\ge 0,\\
  X(t)&=\phi(t), \quad t\in[-\tau,0],
\end{split}
\end{align}
where $\nu$ is a measure in $M([-\tau,0],\R^{d\times d})$,
$\Sigma\in C(\Rp;\R^{d\times d'})$, $f\in C(\Rp;\R^{d})$. For every
$\phi\in C([-\tau,0],\R^d)$ there exists a unique, adapted strong
solution $(X(t,\phi): t\geq-\tau)$ with finite second moments of
\eqref{eq.stoch2} (cf., e.g., Mao \cite{Mao2}). The dependence of
the solution on the initial condition $\phi$ is neglected in our
notation in what follows; that is, we will write $X(t)=X(t,\phi)$
for the solution of \eqref{eq.stoch2}.

Turning our attention to the deterministic equation in $\R^d$
underlying \eqref{eq.stoch2}. For fixed constant $\tau\geq0$:
\begin{align}\label{eq.fund2}
    x'(t) = \int_{[-\tau,0]}  \nu(ds)\, x(t+s) \quad\text{for } t \ge 0,\quad  x(t)=\phi(t) \quad t\in[-\tau,0].
\end{align}
 For every $\phi\in C([-\tau,0],\R^d)$ there is a unique $\mathbb{R}^{d}$-valued
function $x=x(\cdot,\phi)$ which satisfies (\ref{eq.fund2}).

The so-called {\em fundamental solution or resolvent of
\eqref{eq.stoch2}} is the matrix-valued function
$r:\Rp\to\R^{d\times d}$, which is the unique solution of
\begin{align}    \label{eq.fund22}
r'(t) = \int_{[\max\{-\tau,-t\},0]}  \nu(ds)\, r(t+s) \quad\text{for
} t \geq0,\quad  r(0)=\I_d.
\end{align}
For convenience one could set $r(t)=0_{d,d}$ for $t\in[-\tau,0)$.

The solution $x(\cdot,\phi)$ of \eqref{eq.fund2} for an arbitrary
initial segment $\phi$ exists, is unique, and can be represented as
\[
    x(t,\phi) = r(t)\phi(0) + \int_{-\tau}^{0}\int_{[-\tau,u]}\nu(ds)r(t+s-u)\phi(u)du, \quad \text{ for } t\geq0;
\]
cf. Diekmann et al. \cite[Chapter I]{odsgslhw}.

By Rei\ss, Riedle and van Gaans \cite[Lemma 6.1]{rrg} the solution
$(X(t): t\geq-\tau)$ obeys a variation of constants formula:
\begin{align}\label{eq.stochconfor2}
 X(t)=
 \begin{cases}
 x(t) + \int_{0}^{t}r(t-s)f(s)\,ds + \int_0^t r(t-s)\Sigma(s)\,dB(s), &\quad t\geq 0,\\
 \phi(t), &\quad t\in[-\tau,0].
 \end{cases}
\end{align}
The process $X$ defined by \eqref{eq.stochconfor2} obeys
\eqref{eq.stoch2} pathwise on an almost sure event.

In order to determine the asymptotic behaviour of the solution $X$
of \eqref{eq.stoch2}, we argue below, in a very similar manner to
that given in Section 3.1, that the asymptotic behaviour of the
stochastic convolution term on the right--hand side of
\eqref{eq.stochconfor2} can be tackled by identifying the resolvent
$r$ in \eqref{eq.fund22} with the function in \eqref{eq:rR1} and the
convolution term with the integral $Y$ defined by \eqref{eq:Yvar}.

Towards this end, we start by defining the function
$g_\nu:\mathbb{C}\to\C$ by
\[
    g_\nu(\lambda) = \text{det}\left( \lambda I_d - \int_{[-\tau,0]} \e^{\lambda s}\nu(ds) \right).
\]
and also the set of its zeros
\[
   \Lambda=\{\lambda\in\C: g_\nu(\lambda)=0 \}.
\]
The function $g_\nu$ is analytic, and so the elements of $\Lambda$
are isolated. Define
\begin{equation}\label{eq.alpha2}
    \alpha := \sup\{ \Re(\lambda) : g_\nu(\lambda)=0 \}.
\end{equation}
Once again $\alpha$ is finite. Furthermore the cardinality of
$\Lambda'= \{\Re(\lambda)=\alpha : \lambda\in\Lambda \}$ is finite.
Then, following a similar argument as in Subsection~\ref{sub.Vol},
there exists $\varepsilon_0>0$ such that $g_\nu(\lambda)\not=0$ for
$\alpha-\varepsilon_0\leq\Re(\lambda)<\alpha$ and hence
$g_\nu(\lambda)\not=0$ on the line $\Re(\lambda)=\varepsilon$ for
every $\varepsilon\in(0,\varepsilon_0)$. Thus we have
\begin{align}\label{eq.solsumrep}
r(t)e^{-\alpha t}= \sum_{\substack{\lambda_j\in \Lambda',
\Im(\lambda_j)\geq0}} \Big( \tilde{P}_j(t)\cos(\beta_j) t)+
\tilde{Q}_j(t)\sin(\beta_j) t)\Big) + o(e^{-\varepsilon t}), \,
t\to\infty,
\end{align}
where $\Re(\lambda_j)=\alpha$ and $\Im(\lambda_j)=\beta_j$, and
where $\tilde{P}_j$ and $\tilde{Q}_j$ are matrix--valued polynomials
of degree $n_j$, with $n_j+1$ being the order of the pole
$\lambda_j=\alpha+i\beta_j$ of $[g_\nu]^{-1}$.
This is a restatement of Diekmann et
al~\cite[Theorem~5.4]{odsgslhw}.

Let $n$ denote the highest degree of all polynomials associated with
roots in $\Lambda'$ and
let $\lambda_1,...,\lambda_N$ be the finitely many roots in
$\Lambda'$ which have associated polynomials of this degree and have
$\Im(\lambda_j)=\beta_j\geq0$. We associate with each characteristic
root $\lambda_j=\alpha+i\beta_j$ the matrix polynomials $P_j$ and
$Q_{j}$ in \eqref{eq.solsumrep} above, each of which has degree $n$.
 Therefore we may write 
\begin{equation} \label{eq.PjQjstarfin}
    P_j(t)=t^n P_j^* + O(t^{n-1}), \quad Q_j(t)=t^n Q_j^* + O(t^{n-1}).
\end{equation}
where at least one of $P_j^*$ and $Q_j^*$ are not equal to the zero
matrix, for each $j\in\{1,...,N\}$. The precise values of $P_j^*$
and $Q_j^*$ can be determined from the Laurent series of the inverse
of the characteristic function, $g_\nu$, expanded about $\lambda_j$,
c.f. \cite[pp.31]{odsgslhw} i.e.
\begin{equation}\label{eq:tayfK}
    \left[\lambda I_d - \int_{[-\tau,0]}\e^{\lambda s}\nu(ds)\right]^{-1}
    = \sum_{m=0}^{n}\frac{m! \,K_{j,m}}{(\lambda-\lambda_j)^{m+1}} + \hat{q}_j(\lambda),
\end{equation}
where the remainder term $\hat{q}_j(\lambda)$ is analytic at
$\lambda_j$. If $\lambda_j$ is real then $P_j^*=K_{j,n}$, otherwise
$P^*_j:=2\,\Re(K_{j,n})$ and  $Q^*_j:=-2\,\Im(K_{j,n})$. We note
that \eqref{eq:tayfK} defines the value of $n$.

Finally, we define
\begin{equation}\label{eq:rRsfde}
    R(t)=r(t)-\sum_{j=1}^{N}\e^{\alpha t}\{P_j(t)\cos(\beta_j t) + Q_j(t)\sin(\beta_j t)\}, \quad
    t\geq0.
\end{equation}
It is clear that $R$ is a.e. absolutely continuous. Therefore, by
virtue of the decomposition in \eqref{eq:rR1} in the statement of
Proposition~\ref{prop:Y}, if the growth estimates \eqref{eq:Rorder}
and \eqref{eq:R'order} can be established for $R$ defined by
\eqref{eq:rRsfde}, we will be in a excellent position to apply
Proposition~\ref{prop:Y} to the stochastic convolution term on the
right--hand side of \eqref{eq.stochconfor2}. The relevant estimates
will be provided in Lemma~\ref{lm:Rfin}, which is stated in
Section~\ref{sect:actualresults}.

\section{Main Results} \label{sect:actualresults}
In this section, we state the main results of the paper, which
concern the pathwise and mean--square asymptotic behaviour of the
stochastic Volterra and finite delay equations introduced in the
previous section. In order to do so, the decomposition of the
resolvents and variation of constants formulae established in the
previous section must be aligned with the hypothesis of
Proposition~\ref{prop:Y}. The missing ingredient in each of the
proofs is an asymptotic estimate on the remainder terms defined in
equations \eqref{eq:rRvol} and \eqref{eq:rRsfde}, and once these are
supplied, the main results follow directly. In addition, this
section contains a number of remarks on the scope and ramifications
of these main asymptotic results.

We now state the main results for the Volterra equation and affine
SFDE with finite memory.
\begin{theorem}\label{thm:Vol}
    Let $\alpha^*$ and $\alpha$, as defined by \eqref{eq:tranE} and \eqref{eq.alpha1} respectively, obey $\alpha^*<\alpha$.
    Let $n$ be given by \eqref{eq:tayK} (i.e. $n+1$ denotes the highest order of all roots in $\Lambda''=\Lambda\cap \{\Re(\lambda)=\alpha, \Im(\lambda)\geq0\}$) and
let $(\lambda_j)_{j=1}^N$ be the finitely many roots in $\Lambda''$
with this order. Define $\beta_j=\Im(\lambda_j)$, $j=1,\ldots,N$.
Suppose that $P_j^\ast, Q_j^\ast$ for $j=1,\ldots,N$ are given by
\eqref{eq.PjQjstarvol}.
Let $f\in C([0,\infty);\R^{d})$ be such that
\begin{equation}\label{eq:volf}
    \int_{0}^{\infty}\e^{-\alpha t}|f(t)|\,dt<+\infty
\end{equation}
and let $\Sigma\in C([0,\infty);\mathbb{R}^{d\times d'})$ be such
that
\begin{equation}\label{eq:volSig}
\int_0^\infty e^{-2\alpha t}\|\Sigma(t)\|^2 \,dt <+\infty.
\end{equation}
    Let $X$ be the unique solution of \eqref{eq.stoch}.
    Then
    \begin{multline} \label{eq.Xasyvol}
        \lim_{t\to\infty} \left( \frac{X(t)}{t^n\e^{\alpha t}}
         -\sum_{j=1}^{N}\{(Q_j^* X_0 + M_{1,j})\sin(\beta_j t) + (P_j^* X_0 + M_{2,j})\cos(\beta_j t)\} \right)
         \\ =0, \quad \text{a.s.}
    \end{multline}
where $M_{1,j}$ and $M_{2,j}$ are given by Corollary~\ref{cor:Yf}.
\end{theorem}

We are now in a position to prove this Theorem. As already pointed
out, in order to do this, estimates are needed on the asymptotic
behaviour of $R$ defined by \eqref{eq:rRvol}. The following result,
whose proof is deferred to Section~\ref{sect:pflemmaRRprasy} can be
established.
\begin{lemma} \label{lm:Rvol}
Let $R$ be defined by \eqref{eq:rRvol}. Suppose that $\alpha^*$ and
$\alpha$, defined by \eqref{eq:tranE} and \eqref{eq.alpha1}
respectively, obey $\alpha^*<\alpha$. Then there exists
$\varepsilon\in(0,\alpha-\alpha^*)$ such that
\begin{itemize}
\item[(i)] If $n=0$, then  $R(t)=O(e^{(\alpha-\varepsilon)t})$ as $t\to\infty$.
\item[(ii)] If $n=0$, then $R'(t)=O(e^{(\alpha-\varepsilon)t})$ as $t\to\infty$.
\item[(iii)] If $n\geq 1$, then $R(t)=O(t^{n-1}e^{\alpha t})$ as $t\to\infty$.
\item[(iv)] If $n\geq 1$, then
    $R'(t) = O(t^n \e^{\alpha t})$, as $t\to\infty$.
\end{itemize}
\end{lemma}
Therefore, by Lemma~\ref{lm:Rvol}, the function $R$ defined by
\eqref{eq:rRvol} obeys equations \eqref{eq:Rorder} and
\eqref{eq:R'order}. Also a rearrangement of $r$ given by
\eqref{eq:rRvol} yields the form of \eqref{eq:rR1}. Thus, the proof
of Theorem~\ref{thm:Vol} is an immediate consequence of
Lemma~\ref{lm:Rvol}, Corollary~\ref{cor:Yf} and Remark~\ref{rk:det}.

\begin{remark}
The condition $\alpha^*<\alpha$
is imposed as in order to apply Theorem~7.2.1 of \cite{GrLoSt90}, it
is needed that the Laplace transform of $\mu$ in $h_\mu$ is
well--defined over an open region of the complex plane which
contains the critical line $\Re(\lambda)=\alpha$. Theorem~7.2.1 of
\cite{GrLoSt90} then allows one to conclude the asymptotic behaviour
of the deterministic resolvent, \eqref{eq.solsumrepvol}. This
condition is also required in determining the asymptotic behaviour
of the remainder term $R$ of \eqref{eq:rRvol}.

In the case that $\alpha^*=\alpha$ (i.e. the line on which lie the
zeros of $h$ with largest real part co--incides with the
boundary of the region of existence of the Laplace transform of
$|\mu|$), then the deterministic theory differs to that as described
by Theorem~7.2.1 of \cite{GrLoSt90}. The asymptotic behaviour in
this case is examined in great depth in Jordan et
al.~\cite{jsw:1982}, Kriszten and Terj\'eki~\cite{krister:1988} and
Miller~\cite{miller:1974}. In particular, in order to apply
successfully our stochastic admissibility results, we need good
asymptotic information about both the resolvent and its derivative.
For the cases covered here, existing deterministic results for the
resolvent suffice, but new work has been required, and is supplied,
for the derivative. Thus, in this case the stochastic theory as
described by Theorem~\ref{thm:Vol} would not necessarily hold.

Some articles which examine the case when the line containing the
{\it leading} characteristic exponents of the characteristic
equation co--incides with the boundary of the domain of the
transform of the measure are e.g. \cite[Chapter~7.3]{GrLoSt90},
\cite{krister:1988} for deterministic theory and \cite{jasddr:2006},
\cite{jasddr:2007}, for stochastic theory.
\end{remark}

The corresponding result for the affine SFDE is as follows.
\begin{theorem}\label{thm:finite}
    Let $\alpha$ be as defined by \eqref{eq.alpha2}.
    Let $n$ be given by \eqref{eq:tayK} (i.e. $n+1$ denotes the highest order of all roots in $\Lambda''=\Lambda\cap \{\Re(\lambda)=\alpha, \Im(\lambda)\geq0\}$) and
let $(\lambda_j)_{j=1}^N$ be the finitely many roots in $\Lambda''$
with this order.  Define $\beta_j=\Im(\lambda_j)$, $j=1,\ldots,N$.
Suppose that $P_j^\ast, Q_j^\ast$ for $j=1,\ldots,N$ are given by
\eqref{eq.PjQjstarfin}. Let $f\in C([0,\infty);\R^{d})$ be such that
\[
    \int_{0}^{\infty}\e^{-\alpha t}|f(t)|\,dt<+\infty
\]
and let $\Sigma\in C([0,\infty);\mathbb{R}^{d\times d'})$ be such
that
\begin{equation}\label{eq:SigL2fin}
    \int_0^\infty e^{-2\alpha t}\|\Sigma(t)\|^2\,dt <+\infty.
\end{equation}
    Let $X$ be the unique solution of \eqref{eq.stoch2}.
    Then
    \begin{equation} \label{eq.Xasyfin}
        \lim_{t\to\infty} \left( \frac{X(t)}{t^n\e^{\alpha t}}
         -\sum_{j=1}^{N}\{J_{1,j}\sin(\beta_j t) + J_{2,j}\cos(\beta_j t)\} \right) =0, \quad
         \text{a.s.}
    \end{equation}
    where
    \begin{align*}
        J_{1,j} &= Q_j^* \phi(0)+ G_{1,j} + M_{1,j}, \quad  J_{2,j} = P_j^* \phi(0)+ G_{2,j} + M_{2,j}, \\
        G_{1,j} &= \int_{-\tau}^{0}\int_{[-\tau,u]}\e^{\alpha u}\nu(ds)\{Q_j^* \cos(\beta_j u) - P_j^* \sin(\beta_j u) \} \phi(s-u) du, \\
        G_{2,j} &= \int_{-\tau}^{0}\int_{[-\tau,u]}\e^{\alpha u}\nu(ds)\{P_j^* \cos(\beta_j u) + Q_j^* \sin(\beta_j u)\} \phi(s-u) du,
    \end{align*}
    and
    where $M_{1,j}$ and $M_{2,j}$ are given by Corollary~\ref{cor:Yf}.
\end{theorem}
As in the case of Theorem~\ref{thm:Vol}, we are now ready to prove
this Theorem. As indicated earlier, we can do this once appropriate
estimates are available for the asymptotic behaviour of $R$ defined
by \eqref{eq:rRsfde}. These estimates are supplied in the following
result, whose proof is deferred to
Section~\ref{sect:pflemmaRRprasy}.
\begin{lemma} \label{lm:Rfin}
Let $R$ be defined by \eqref{eq:rRsfde}. Suppose that $\alpha$ is as
defined by \eqref{eq.alpha2}. Then there exists $\varepsilon>0$ such
that
\begin{itemize}
\item[(i)] If $n=0$, then $R(t)=O(e^{(\alpha-\varepsilon)t})$ as $t\to\infty$.
\item[(ii)] If $n=0$, then $R'(t)=O(e^{(\alpha-\varepsilon)t})$ as $t\to\infty$.
\item[(iii)] If $n\geq 1$, then $R(t)=O(t^{n-1}e^{\alpha t})$ as $t\to\infty$.
\item[(iv)] If $n\geq 1$, then
    $R'(t) = O(t^n \e^{\alpha t})$, as $t\to\infty$.
\end{itemize}
\end{lemma}
Observe that Lemma~\ref{lm:Rfin} states that $R$ defined by
\eqref{eq:rRsfde} obeys equations \eqref{eq:Rorder} and
\eqref{eq:R'order}. Also a rearrangement of $r$ given by
\eqref{eq:rRsfde} yields the form of $r$ in \eqref{eq:rR1}. Thus,
the proof of Theorem~\ref{thm:finite} is an immediate consequence of
Lemma~\ref{lm:Rfin}, Corollary~\ref{cor:Yf} and Remark~\ref{rk:det}.

\begin{remark}
    Theorem~\ref{thm:finite} differs from Theorem~\ref{thm:Vol} with respect to the region of existence of the characteristic equation $g_{\nu}$, i.e. $\int_{[-\tau,0]}\e^{a s}|\nu|(ds)$ exists for all $a\in(-\infty,\infty)$ and thus the condition $\alpha^*<\alpha$, present in Theorem~\ref{thm:Vol}, has no analogue in Theorem~\ref{thm:finite}.
\end{remark}
\begin{remark}
While Theorems~\ref{thm:Vol} and \ref{thm:finite} give a rate of
growth or decay in an almost sure sense, it is observed, via
Theorem~\ref{thm.ascharacterise2}, that this convergence also holds
in mean square. That is, for the solution of the Volterra equation
\eqref{eq.stoch}, with the assumptions of Theorem~\ref{thm:Vol},
\begin{align*}
\lim_{t\to\infty}\mathbb{E}\left[\norm{\frac{X(t)}{t^n\e^{\alpha t}}
         -\sum_{j=1}^{N}\{(Q_j^* X_0 + M_{1,j})\sin(\beta_j t) + (P_j^* X_0 + M_{2,j})\cos(\beta_j t)\}}^2 \right] =0.
\end{align*}
Also, for the solution of the finite delay equation
\eqref{eq.stoch2}, with the assumptions of Theorem~\ref{thm:finite},
\[
\lim_{t\to\infty}\mathbb{E}\left[\norm{\frac{X(t)}{t^n\e^{\alpha t}}
         -\sum_{j=1}^{N}\{J_{1,j}\sin(\beta_j t) + J_{2,j}\cos(\beta_j t)\}}^2 \right] =0.
\]
\end{remark}
\begin{remark}\label{rk:det}
    The asymptotic behaviour of the deterministic functional differential equations \eqref{eq.fund1} or
    \eqref{eq.fund22}, each of which obey 
    \begin{equation}\label{eq:asyr}
        \lim_{t\to\infty} \left( \frac{r(t)}{t^n\e^{\alpha t}}
         -\sum_{j=1}^{N}\{Q_j^*\sin(\beta_j t) + P_j^*\cos(\beta_j t)\} \right) =0,
    \end{equation}
    where $P_j^\ast$ and $Q_j^\ast$ are determined by \eqref{eq.PjQjstarvol} (in the case of the Volterra equation)
    and \eqref{eq.PjQjstarfin} (for the equation with finite delay) is analogous to the asymptotic behaviour of $X$ as given by \eqref{eq.Xasyvol} and \eqref{eq.Xasyfin} respectively.

  It can therefore be seen, despite the presence of the stochastic integral, that $X$ inherits the asymptotic behaviour of $r$, provided
  that the intensity of the noise perturbation does not grow too rapidly.

    Regarding the multipliers of the trigonometric terms we remark that $M_{1,j}$ and $M_{2,j}$ are Gaussian distributed random
    variables and hence their values and, in particular, sign will depend upon the sample path. Moreover these random variables
    depend on the coefficients of the trigonometric terms in \eqref{eq:asyr} i.e. $P_j^*$ and $Q_j^*$.
\end{remark}

\begin{remark}\label{rk:fsignec}
    The conditions \eqref{eq:SigL2} and \eqref{eq:fL1} on the growth of $\Sigma$ and $f$ are, in some sense, unimprovable if the asymptotic behaviour of $X$ is to be recovered.

Consider, for example, the scalar ordinary affine stochastic
equation
\[
    dX(t)= \bigl( \alpha X(t) + f(t)\bigr)\,dt + \Sigma(t)\,dB(t), \quad t\geq0, \quad X(0)=X_0\in\R,
\]
where $\alpha\in\R$, $\Sigma\in C([0,\infty);\R)$ and $f$ is a
non--negative function, i.e. $f\in C([0,\infty);[0,\infty))$. Then
we have the following equivalent conditions:
\begin{itemize}
\item[(i)] \eqref{eq:SigL2} and \eqref{eq:fL1} hold.
\item[(ii)] There exists an a.s. finite random variable $L$ such that
\begin{equation}\label{eq:Mconv}
    \mathbb{P}\left[ \lim_{t\to\infty} \e^{-\alpha t}X(t) = L\in(-\infty,\infty) \right]>0.
\end{equation}
\item[(iii)] There exists an a.s. finite random variable $L$ such that
\begin{equation}\label{eq:Mconv2}
\lim_{t\to\infty} \e^{-\alpha t}X(t) = L, \quad \text{a.s.}
\end{equation}
\end{itemize}
The proof of Remark~\ref{rk:fsignec} is deferred to
Section~\ref{sect:remarkpf}.
\end{remark}

\begin{remark}
    The asymptotic behaviour of the solution of \eqref{eq.stoch2} in the case when $\alpha<0$ and the diffusion coefficient is time independent, i.e. $\Sigma(t) = \Sigma\in\mathbb{R}^{d\times d'}$ for all $t\geq0$, is considered in \cite{appmaowu}.
    It is argued that asset prices in financial markets fluctuate and therefore it is of interest to describe the order of the oscillations
    about the mean in particular the rate of growth of the running maximum of this asset price.
    In this case the resolvent function decays exponentially to zero resulting in the process $X$
    behaving asymptotically like a Gaussian process. Specifically, it is shown that
    \[
        \limsup_{t\to\infty}\frac{|X(t)|_{\infty}}{\sqrt{2\log t}}
        = \max_{i=1,...,d} \sqrt{\sum_{k=1}^{m}\biggr( r(s)\Sigma \biggr)_{i,k}^2 ds}, \quad \text{a.s.}
    \]
    However for constant coefficient of diffusion, condition \eqref{eq:SigL2fin} is violated
    and hence Theorem~\ref{thm:finite} does not apply.
\end{remark}
\begin{remark}
    The asymptotic behaviour of the solution of the scalar equation \eqref{eq.stoch2}, with $d=1$, is considered
 in \cite{jamrcs10} with $\alpha\geq0$, the zero of $g$ which has
this real part is a simple real zero and all other zeros of $g$ have
real parts less than $\alpha$. Thus
\cite[Theorem~3.1~(b)]{jamrcs10}, which considers the case of
$\alpha>0$, is a special case of Theorem~\ref{thm:finite}. Moreover,
as in practice it is quite difficult to determine the zeroes of $g$
a subclass of measures is looked at which give the desired
properties on the zeroes of $g$. Also, the economic interpretations
of these impositions are discussed.
    To summarise the results: it is shown that if $\alpha=0$ then the market behaves similar to a Black-Scholes model,
    in particular $X$ undergoes fluctuations according to the law of the iterated logarithm.
    \[
        \limsup_{t\to\infty}\frac{X(t)}{\sqrt{2t\log\log t}} = -\liminf_{t\to\infty}\frac{X(t)}{\sqrt{2t\log\log t}} = C_1,
    \]
    where $C_1$ is a positive constant. On the other hand, the case $\alpha>0$ gives
    \[
        \lim_{t\to\infty} \e^{-\alpha t}X(t) = C_2,
    \]
    where $C_2$ is a random variable. This regime is interpreted as the market undergoing a bubble or crash,
    depending upon the sign of $C_2$, with both events being possible.

    However the case $\alpha=0$ studied in \cite{jamrcs10} also has a constant diffusion coefficient, thus \eqref{eq:SigL2fin} is not satisfied and so Theorem~\ref{thm:finite} does not apply.
\end{remark}

\section{Examples}\label{sect:examples}
We give some illustrative examples of Theorems~\ref{thm:Vol} and
~\ref{thm:finite} and Proposition~\ref{prop:Y}.
The first three examples consider the situation where the resolvent
is of the especially simple form
\[
\mu(ds)=A\delta_{0}(ds),
\]
where $A$ is a $d\times d$ matrix with real entries. In this case,
the resolvent is nothing other than the principal matrix solution
\[
r'(t)=Ar(t), \quad r(0)=I_d
\]
and the stochastic equation is just the affine stochastic
differential equation
\[
dX(t)=AX(t)\,dt+\Sigma(t)\,dB(t), \quad t\geq 0; \quad X(0)=\xi.
\]
Since there are no more than $d$ eigenvalues, the resolvent $r$ and
its derivative can be expressed as finite sums, and so there is no
need for a detailed analysis of remainder terms.

Our first example looks at the case when the leading eigenvalue (or
zero of the characteristic equation) has algebraic multiplicity
equal to the geometric multiplicity.
\begin{example}
Suppose that $A=\gamma I$ where $I$ is the $2\times 2$ identity
matrix. Then $Y(t)=\e^{-\gamma t}X(t)$ obeys $dY(t)=\e^{-\gamma t}
\Sigma(t)\,dB(t)$, so
\[
Y(t)=\xi + \int_0^t  \e^{-\gamma s}\Sigma(s)\,dB(s), \quad t\geq 0.
\]
In this case, applying our results to $Y$, we have $\alpha=0$. If
$s\mapsto \e^{-\gamma s}\Sigma(s)\in L^2(0,\infty)$, by the
martingale convergence theorem we have
\[
\lim_{t\to\infty} \frac{X(t)}{\e^{\gamma t}} =\lim_{t\to\infty}
Y(t)=\xi + \int_0^\infty   \e^{-\gamma s}\Sigma(s)\,dB(s), \quad
\text{a.s.}
\]
Let $\lambda_j=0$. Since $A-\gamma I=0$, we see that, with $n=0$,
$K_{j,0}=I$ and $\hat{q}_j(\lambda)=0$, we have
\[
\left(\lambda I -(A-\gamma I)\right)^{-1}=\lambda^{-1} I =
\sum_{m=0}^n \frac{m!\, K_{j,m}}{\lambda^{m+1}}+\hat{q}_j(\lambda).
\]
Thus, we may set $P_j^\ast=I$, and therefore the limit for $Y$ has
the form predicted by Theorem~\ref{thm:finite} with $\alpha=0$.
\end{example}

We now demonstrate the resulting asymptotic behaviour of the
solution of the stochastic equation when the leading eigenvalue has
geometric multiplicity less than the algebraic multiplicity.
\begin{example}
Suppose that
\[
A=\left(
\begin{array}{cc}
\gamma & 1 \\
0 & \gamma
\end{array}
\right).
\]
Consider $Y(t)=\e^{-\gamma t}X(t)$. Then
\[
dY(t)=(A-\gamma I)Y(t)\,dt + \e^{-\gamma t}\Sigma(t)\,dB(t).
\]
Then, applying our theory to $Y$, we find that $\alpha=0$, because
$\lambda=0$ is an eigenvalue of multiplicity 2. In this case $r$ is
given by
\[
r(t)=\left(
\begin{array}{cc}
1 & t \\
0 & 1
\end{array}
\right).
\]
Since $\det(r(t))=1$ for all $t\geq 0$, $r(t)$ is invertible, and we
may write $r(t-s)=r(t)r^{-1}(s)=r(t)r(-s)$ for all $0\leq s\leq t$.
Therefore
\[
Y(t)=r(t)\xi + \int_0^t r(t-s)\e^{-\gamma s}\Sigma(s)\,dB(s)=r(t)\xi
+ r(t)\int_0^t r(-s)\e^{-\gamma s}\Sigma(s)\,dB(s).
\]
Notice that $r(t)=I_{d}+t(A-\gamma I)$ and $(A-\gamma
I)r(-s)=A-\gamma I$. Then
\begin{align*}
\lefteqn{\frac{r(t)}{t}\int_0^t r(-s) \e^{-\gamma s}\Sigma(s)\,dB(s)}\\
 &= \frac{1}{t}\int_0^t r(-s)\Sigma(s)\,dB(s)+\int_0^t (A-\gamma I)r(-s) \e^{-\gamma s}\Sigma(s)\,dB(s)\\
  &= \frac{1}{t}\int_0^t r(-s) \e^{-\gamma s}\Sigma(s)\,dB(s)+(A-\gamma I)\int_0^t \e^{-\gamma s}\Sigma(s)\,dB(s).
\end{align*}
Using Lemma~\ref{lemma.tminkskfto0}, the first term has zero limit
as $s\mapsto \e^{-\gamma s}\Sigma(s)$ is in $L^2(0,\infty)$, and
$r(-s)/s\to -(A-\gamma I)$ as $s\to\infty$. The second term
converges by the martingale convergence theorem. Thus
\[
\lim_{t\to\infty} \frac{X(t)}{t\e^{\gamma t}}=(A-\gamma I)\xi +
(A-\gamma I)\int_0^\infty \e^{-\gamma s}\Sigma(s)\,dB(s).
\]
This is exactly the form of the limit predicted in
Theorem~\ref{thm:finite}, because for $\lambda_j=0$ with $n=1$, we
have
\[
P_j^\ast= K_{j,1}=\lim_{\lambda \to 0} \lambda^2 \left(\lambda
I_d-(A-\gamma I)\right)^{-1}=A-\gamma I.
\]
\end{example}

This next example demonstrates the case when the leading eigenvalues
are complex solutions of the characteristic equation.
\begin{example}
Suppose that
\[
A=\left(
\begin{array}{cc}
\gamma  & -1 \\
1 & \gamma
\end{array}
\right).
\]
Suppose that $Y(t)=\e^{-\gamma t}X(t)$. If $J=A-\gamma I$, then
\[
dY(t)=JY(t)\,dt + \e^{-\gamma t}\Sigma(t)\,dB(t).
\]
For the equation solved by $Y$, we have $\alpha=0$, because
$\lambda=\pm i$ are eigenvalues of multiplicity 1. In this case $r$
is given by
\[
r(t)=\left(
\begin{array}{cc}
\cos(t) & -\sin(t) \\
\sin(t) & \cos(t)
\end{array}
\right).
\]
Since $\det(r(t))=1$ for all $t\geq 0$, $r(t)$ is invertible, and we
may write $r(t-s)=r(t)r^{-1}(s)=r(t)r(-s)$ for all $0\leq s\leq t$.
Therefore
\[
X(t)=r(t)\xi + \int_0^t r(t-s)\e^{-\gamma s}\Sigma(s)\,dB(s)=r(t)\xi
+ r(t)\int_0^t r(-s)\e^{-\gamma s}\Sigma(s)\,dB(s).
\]
Since $r(-s)$ is bounded, and $s\mapsto e^{-\gamma s}\Sigma(s)\in
L^2(0,\infty)$, it follows that
\[
\lim_{t\to\infty} \int_0^t r(-s)\e^{-\gamma s}\Sigma(s)\,dB(s) =
\int_0^\infty r(-s)\e^{-\gamma s}\Sigma(s)\,dB(s), \quad \text{a.s.}
\]
Therefore
\[
\lim_{t\to\infty}\left\{Y(t)-r(t)\left(\xi + \int_0^\infty
r(-s)\Sigma(s)\,dB(s)\right)\right\}=0,
 \quad \text{a.s.}
\]
We now see that $r(t)=\cos(t)I + \sin(t)J$, and so the following
limit holds almost surely:
\begin{multline*}
\lim_{t\to\infty}\left\{\frac{X(t)}{\e^{\gamma
t}}-\left(\cos(t)I+\sin(t)J\right)\left(\xi + \int_0^\infty
(\cos(s)I -\sin(s)J)\e^{-\gamma
s}\Sigma(s)\,dB(s)\right)\right\}\\=0.
\end{multline*}
Setting
\[
G_c=\int_0^\infty \cos(s)\e^{-\gamma s}\Sigma(s)\,dB(s), \quad
G_s=\int_0^\infty \sin(s)\e^{-\gamma s}\Sigma(s)\,dB(s)
\]
and noting that $J^2=-I$, we get
\begin{equation*}
\lim_{t\to\infty}\Biggl\{\frac{X(t)}{\e^{\gamma
t}}-\cos(t)\left(\xi+G_c-JG_s \right) - \sin(t)\left(J\xi+JG_c+
G_s\right) \Biggr\} =0,  \quad \text{a.s.}
\end{equation*}
To show that this asymptotic expansion agrees exactly with formula
\eqref{eq.Xasyfin} derived in Theorem~\ref{thm:finite} we notice for
$\lambda_j=(-1)^{j-1} i$ for $j=1,2$ where each of which has
multiplicity $n+1=1$, that
\[
K_{j,0}=\lim_{\lambda\to \lambda_j} (\lambda-\lambda_j)\left(\lambda
I- J\right)^{-1} =\lim_{\lambda\to\lambda_j}
(\lambda-\lambda_j)\frac{1}{1+\lambda^2} \left(
\begin{array}{cc}
\lambda & -1 \\
1 & \lambda
\end{array}
\right).
\]
Since
$(\lambda-\lambda_j)(\lambda-\overline{\lambda}_j)=1+\lambda^2$, we
have
\[
K_{j,0}= \frac{1}{\lambda_j-\overline{\lambda}_j} \left(
\begin{array}{cc}
\lambda_j & -1 \\
1 & \lambda_j
\end{array}
\right) = \frac{1}{2\lambda_j} \left(
\begin{array}{cc}
\lambda_j & -1 \\
1 & \lambda_j
\end{array}
\right) = \frac{1}{2} \left(
\begin{array}{cc}
1 & \lambda_j \\
-\lambda_j & 1
\end{array}
\right)
\]
Hence $2K_{1,0}=I-iJ$ and $2K_{2,0}=I+iJ$. Therefore $P_1^\ast=I$
and $Q_1^\ast=J$.
\end{example}
We provide an example of a convolution Volterra
integro--differential equation where the zeros of the characteristic
equation do not lie in the domain of the transform of the measure,
i.e. $\alpha^\ast>\alpha$. Nevertheless an explicit formula for the
resolvent may obtained and hence one may deduce the asymptotic
behaviour of the solution of the stochastic equation.
\begin{example}
Let $X$ be the unique solution of
\[
    dX(t) = \int_{[0,t]}\mu(ds)X(t-s) dt + \Sigma(t) dB(t), \quad t\geq0
\]
where $X(0)=X_0\in\R^d$ and $\mu(ds) = -6\,\delta_0(ds)I_d -
4\,\e^{-s} \,dsI_d$. Hence $\alpha^\ast=-1$ and $h$ is given by
\[
    h(\lambda) = \det\left(\lambda I_d - \int_{[0,\infty)}\mu(ds)\e^{-\lambda s}I_d \right)
    = \frac{(\lambda+2)^d (\lambda+5)^d}{(\lambda+1)^{d}}.
\]
Thus $\alpha^\ast=-1>-2=\alpha$ and so we cannot apply
Theorem~\ref{thm:Vol} to this problem.

Nevertheless, the differential resolvent, \eqref{eq.fund1}, may
rewritten as the solution of a second order equation and solved to
give
\[
    r(t) = -\frac{1}{3}\e^{-2t} I_d + \frac{4}{3}\e^{-5t} I_d.
\]
Therefore $n=0$ and $P_1^*=-1/3$ and one can now apply
Proposition~\ref{prop:Y} to determine the asymptotic behaviour of
$X$, i.e.
\[
    \lim_{t\to\infty}\frac{X(t)}{\e^{-2t}} = -\frac{1}{3}X_0 -\frac{1}{3}\int_{0}^{\infty}\e^{2s}\Sigma(s)dB(s).
\]
Thus, in instances where Theorem~\ref{thm:Vol} does not apply,
providing that the asymptotic behaviour of $r$ may be estimated to
agree with \eqref{eq:rR1}, then via Proposition~\ref{prop:Y} the
asymptotic behaviour of the solution of the stochastic equation can
still be recovered.
\end{example}

We finish with an example where the underlying deterministic
functional differential equation is not equivalent to a linear
ordinary differential equation, but for which it is possible, owing
to the special structure of the equation, to determine exactly the
leading order asymptotic behaviour.
\begin{example}
Suppose that $X$ obeys
\[
dX(t)=a(X(t)-X(t-1/3))\,dt + \Sigma(s)\,dB(t), \quad t\geq 0,
\]
where $\Sigma\in C(\Rp;\mathbb{R}^{1\times d'})$, $X(t)=\phi(t)$ for
$t\in [-1/3,0]$, where $\phi\in C([-1/3,0],\R)$. Let
$a=3/(1-1/\e)>0$. This is equivalent to choosing $\tau=1/3$ and the
finite measure $\nu(ds)=a\delta_{0}(ds)-a \delta_{-1/3}(ds)$. Then
it can be shown that $\nu([-t,0])\geq 0$ for all $t\in [0,1/3]$ with
$\nu([-1/3,0])=0$. Also
\[
\int_{[-1/3,0]} s\,\nu(ds)=\frac{1}{1-\e^{-1}}>1.
\]
Consequently, all the conditions of part (i), Theorem~3.3 in
\cite{jamrcs10} hold, and therefore there is a unique positive real
solution $\lambda_1>0$ of $g_\nu(\lambda_1)=0$ where
$g_\nu(\lambda)=\lambda-a+a\e^{-\lambda/3}$, and moreover
$\alpha=\lambda_1$. Since $a=3/(1-1/\e)$, it is easily verified that
$\alpha=\lambda_1=3$. Furthermore, as $g_\nu'(\lambda_1)=1-a
\e^{-1}/3\neq 0$, it can be shown that $n=0$ in
Theorem~\ref{thm:finite}, and moreover by l'H\^{o}pital's rule that
\[
P_1^\ast=\lim_{\lambda\to \lambda_1}
\frac{\lambda-\lambda_1}{g_\nu(\lambda)}
=\frac{1}{g_\nu'(3)}=\frac{1-\e^{-1}}{1-2\e^{-1}}.
\]
Therefore, assuming \eqref{eq:SigL2fin} holds, then all the
conditions of Theorem~\ref{thm:finite} apply, we have that
\[
\lim_{t\to\infty}\frac{X(t)}{\e^{3t}}= P_1^*\phi(0) +
P_1^*\int_{-\tau}^{0}\int_{[-\tau,u]}\e^{3u}\nu(ds)\phi(s-u)du
+P_1^*\int_{0}^{\infty}\e^{-3 s}\Sigma(s) dB(s).
\]
\end{example}

\section{Proofs of Supporting Results}}\label{sect:proof3}
This section contains the proofs of some supporting results: the
first part of this section concerns the variation of constants
formula \eqref{eq.stochconfor} in
Proposition~\ref{prop.stochresrep}; the rest of the section is
devoted to the a.s. convergence to zero of stochastic integrals
whose integrands involve $t$--dependence, but have special features.
\subsection{Proof of Proposition~\ref{prop.stochresrep}}
Define $w$ to be the unique continuous solution of
\[
    w'(t) = f(t) + \int_{[0,t]}\mu(ds)w(t-s), \quad t\geq0, \quad w(0)=0.
\]
Then $w(t)=\int_{0}^{t}r(t-s)f(s)\,ds$ for $t\geq 0$. Noting that
$X$ is the unique continuous adapted process which obeys
\eqref{eq.stoch}, we may defined the continuous adapted process
$Z=\{Z(t):t\geq 0\}$ by $Z(t):= X(t)-w(t)$ for $t\geq0$. Then $Z$ is
a semimartingale, and is represented by
\[
    dZ(t) = \int_{[0,t]}\mu(ds)Z(t-s) + \Sigma(t)dB(t), \quad t\geq0, \quad Z(0)=X_0.
\]
Hence Rei\ss, Riedle and van Gaans \cite[Lemma 6.1]{rrg} gives
\[
    Z(t) = r(t)X_0 + \int_{0}^{t}r(t-s)\Sigma(s)dB(s), \quad\text{$\P$-a.s.} \quad t\geq0,
\]
which rearranges to yield \eqref{eq.stochconfor}.

\subsection{Stochastic limit results}
Parts of the proofs of our main results involve the proof of some
subsidiary results which may themselves be of independent interest.
They are stated and proven here. We start with the proof of a
preliminary lemma, which will be used in the proof of
Proposition~\ref{prop:Y}.
\begin{lemma} \label{lemma.tminkskfto0}
Suppose $f\in L^2([0,\infty),\R^{d\times r})$. If $k>0$, then
\[
\lim_{t\to\infty} \frac{1}{(1+t)^k}\int_0^t s^k f(s)\,dB(s)=0,
\quad\text{a.s.}
\]
\end{lemma}
\begin{proof}
Define
\[
K(t)=\frac{1}{(1+t)^k}\int_0^t s^k f(s)\,dB(s), \quad t\geq 0.
\]
Then $dK(t)=-k(1+t)^{-1}K(t)\,dt + (1+t)^{-k}t^kf(t)\,dB(t)$. Hence
for $i=1,\ldots,d$ with $K_i(t):=\langle K(t),\mathbf{e}_i\rangle$,
we have
\[
dK_i(t)=-k(1+t)^{-1}K_i(t)\,dt + \sum_{j=1}^r \frac{t^k}{(1+t)^k}
f_{ij}(t)\,dB_j(t).
\]
Therefore
\begin{multline*}
d\|K(t)\|^2=\left(-2k(1+t)^{-1}\|K(t)\|^2 +
\frac{t^{2k}}{(1+t)^{2k}} \|f(t)\|^2_F \right)\,dt
\\+\sum_{i=1}^d 2K_i(t)\sum_{j=1}^r \frac{t^k}{(1+t)^k} f_{ij}(t)\,dB_j(t).
\end{multline*}
Now define the non--decreasing processes $A_1$ and $A_2$ by
\[
A_1(t)= \int_0^t \frac{s^{2k}}{(1+s)^{2k}} \|f(s)\|^2_F\,ds, \quad
A_2(t)= \int_0^t 2k(1+s)^{-1}\|K(s)\|^2\,ds.
\]
and the martingale $M$ by
\[
M(t)= \sum_{j=1}^r \int_0^t \sum_{i=1}^d 2K_i(s)\frac{s^k}{(1+s)^k}
f_{ij}(s)\,dB_j(s).
\]
Then we have
\[
\|K(t)\|^2 = A_1(t)-A_2(t)+M(t), \quad t\geq 0.
\]
Since $f$ is in $L^2(0,\infty)$, we notice that $A_1(t)$ tends to a
finite limit as $t\to\infty$. Therefore, we have that $\|K(t)\|^2\to
\kappa$ as $t\to\infty$ a.s where $\kappa \in [0,\infty)$ a.s. (It
is known that $\lim_{t\to\infty}\|K(t)\|^2$ exists and is finite
due to \cite[Theorem~7, pp.139]{LipShir}). Then by l'H\^{o}pital's
rule we have
\[
\lim_{t\to\infty} \frac{A_2(t)}{\log t} = 2k\kappa.
\]
Notice now that $M$ has quadratic variation
\[
\langle M\rangle(t)= \int_0^t \sum_{j=1}^r \left(\sum_{i=1}^d
2K_i(s)\frac{s^k}{(1+s)^k} f_{ij}(s)\right)^2\,ds.
\]
Therefore by the Cauchy--Schwartz inequality
\begin{align*}
\langle M\rangle(t) \leq  \int_0^t \sum_{j=1}^r 4\sum_{l=1}^d
K_l^2(s)\sum_{i=1}^d \frac{s^{2k}}{(1+s)^{2k}} f_{ij}^2(s)\,ds \leq
4\int_0^t \|K(s)\|^2 \|f(s)\|^2_F\,ds.
\end{align*}
Since $f$ is in $L^2(0,\infty)$, we see that $\lim_{t\to\infty}
\langle M \rangle(t)$ is finite and hence that $M$ tends to a finite
limit a.s. Let $A=\{\omega:\kappa(\omega)>0\}$ and suppose that
$\mathbb{P}[A]>0$. Then on $A$ we have $\lim_{t\to\infty}
\|K(t,\omega)\|^2=-\infty$, which is a contradiction. Hence
$\mathbb{P}[A]=0$, or $\kappa=0$ a.s. Therefore $K(t)\to 0$ as
$t\to\infty$, a.s., as required.
\end{proof}

The proof of Proposition~\ref{prop:Y}, in the case $n=0$, uses
Lemma~3 from Appleby \cite{ja:2003}; this lemma is used in the proof
of the next supporting convergence result, so is stated for
completeness.
\begin{lemma}\label{lm:ja2003}
    Suppose $x:\Rp\to\Rp$ is a continuous, integrable function, and $\eta>0$ is any fixed constant.
    Then, the sequence $\{a_n\}_{n=0}^{\infty}$ given by $a_0=0$ and
    \[
        a_{n+1} = \inf\left\{t\in[a_n+\eta/2,a_n+3\eta/4]: x(t)=\min_{a_n+\eta/2\leq\tau\leq a_n+3\eta/4}x(\tau)\right\},
        \quad n\in\mathbb{Z}^+,
    \]
satisfies
    \[
        \frac{\eta}{4}<a_{n+1} - a_n < \eta \quad \text{for all $n\in\mathbb{Z}^+$}, \quad \lim_{n\to\infty}a_n=\infty,
    \]
    together with
    \[
        \sum_{n=0}^{\infty}x(a_n)<\infty.
    \]
\end{lemma}
The following lemma which will be used in the proof of
Proposition~\ref{prop:Y} ($n=0$), is a mild adaptation of Lemma~5.2
from \cite{jamrcs10}.
\begin{lemma}\label{lm:kL2}
    Let $k:\Rp\to\R$ be such that $k,k'\in L^2([0,\infty);\R)$.
    Define for $f\in L^2([0,\infty);\R)$ the Gaussian process $\{K(t):t\geq0\}$ by
    \[
        K(t) = \int_{0}^{t}k(t-s)f(s)\,dB(s).
    \]
    Then $\lim_{t\to\infty}K(t)=0$, a.s.
\end{lemma}
\begin{proof}
    We re--express $K$, using the stochastic Fubini Theorem, e.g. \cite[Theorem~4.6.64, pp.210--211]{prot}, which leads to 
    \begin{align*}
        K(t) &= \int_{0}^{t}\left(k(0) + \int_{0}^{t-s}k'(u)\,du\right)f(s)\,dB(s) \\
        &= \int_{0}^{t}k(0)f(s)\,dB(s) + \int_{0}^{t}\int_{s}^{t}k'(v-s)\,dv\,f(s)\,dB(s) \\
        &= k(0)\int_{0}^{t}f(s)\,dB(s) + \int_{0}^{t}\int_{0}^{v}k'(v-s)f(s)\,dB(s)\,dv.
    \end{align*}
    Then for any increasing sequence $\{a_n\}_{n=0}^{\infty}$ we have, for $t\in[a_n,a_{n+1})$,
    \[
        K(t) = K(a_n) + k(0)\int_{a_n}^{t}f(s)\,dB(s) + \int_{a_n}^{t}\int_{0}^{v}k'(v-s)f(s)\,dB(s)\,dv.
    \]
    Squaring, taking suprema and finally an expectation across this inequality gives
    \begin{align}
        \mathbb{E}\left[\sup_{a_n\leq t\leq a_{n+1}}|K(t)|^2\right]
        &\leq 3\,\mathbb{E}\left[K(a_n)^2\right]
        + 3\,k(0)^2\,\mathbb{E}\left[\sup_{a_n\leq t\leq a_{n+1}}\left|\int_{a_n}^{t}f(s)\,dB(s)\right|^2\right] \notag \\
        &\qquad+ 3\,\mathbb{E}\left[\sup_{a_n\leq t\leq a_{n+1}}\left|\int_{a_n}^{t}\int_{0}^{v}k'(v-s)f(s)\,dB(s)\,dv\right|^2\right].
        \label{eq:EK}
    \end{align}
We consider each term on the right--hand side separately. Now for
the second term, applying Doob's inequality, c.f. e.g.
\cite[Theorem~1.38]{Mao2} yields
\[
    \mathbb{E}\left[\sup_{a_n\leq t\leq a_{n+1}}\left|\int_{a_n}^{t}f(s)\,dB(s)\right|^2\right] \leq 4 \int_{a_n}^{a_{n+1}}f(s)^2\,ds
\]
and thus
\begin{equation}\label{eq:t3sum}
    \sum_{n=0}^{\infty}\mathbb{E}\left[\sup_{a_n\leq t\leq a_{n+1}}\left|\int_{a_n}^{t}f(s)\,dB(s)\right|^2\right] <+\infty.
\end{equation}
For the third term, applying the Cauchy--Schwarz inequality gives
\begin{align*}
    \mathbb{E}\biggl[\sup_{a_n\leq t\leq a_{n+1}}&\left|\int_{a_n}^{t}\int_{0}^{v}k'(v-s)f(s)\,dB(s)\,dv\right|^2\biggr] \\
    &\qquad \leq \mathbb{E}\left[\sup_{a_n\leq t\leq a_{n+1}} (t-a_n) \int_{a_n}^{t}\left|\int_{0}^{v}k'(v-s)f(s)\,dB(s)\right|^2\,dv\right] \\
    &\qquad =(a_{n+1}-a_n) \int_{a_n}^{a_{n+1}}\mathbb{E}\left[\left|\int_{0}^{v}k'(v-s)f(s)\,dB(s)\right|^2\right]dv \\
    &\qquad =(a_{n+1}-a_n) \int_{a_n}^{a_{n+1}}\int_{0}^{v}k'(v-s)^2f(s)^2\,ds\,dv.
\end{align*}
Now suppose that $0<a_{n+1}-a_n<\eta$ for some $\eta>0$, then
\begin{align}
    \sum_{n=1}^{\infty}\mathbb{E}\biggl[\sup_{a_n\leq t\leq a_{n+1}}&\left|\int_{a_n}^{t}\int_{0}^{v}k'(v-s)f(s)\,dB(s)\,dv\right|^2\biggr] \notag \\
    &\leq \eta\sum_{n=1}^{\infty}\int_{a_n}^{a_{n+1}}\int_{0}^{v}k'(v-s)^2f(s)^2\,ds\,dv <+\infty. \label{eq:t2sum}
\end{align}
Now the first term, $t\mapsto x(t)=\mathbb{E}[K(t)^2]$, is
continuous and non--negative, also
\[
    \int_{0}^{\infty}x(t)\,dt = \int_{0}^{\infty}k(t)^2\,dt \int_{0}^{\infty}f(s)^2\,ds <+\infty.
\]
Therefore by Lemma~\ref{lm:ja2003}, for all $\eta>0$ there exists a
sequence $\{a_n\}_{n=0}^{\infty}$ such that
\begin{equation}\label{eq:t1sum}
    \sum_{n=0}^{\infty}x(a_n) = \sum_{n=0}^{\infty}\mathbb{E}[K(a_n)^2] <+\infty.
\end{equation}
So, using \eqref{eq:t3sum}, \eqref{eq:t2sum} and \eqref{eq:t1sum} in
\eqref{eq:EK} yields
\[
    \sum_{n=0}^{\infty}\mathbb{E}\left[\sup_{a_n\leq t\leq a_{n+1}}|K(t)|^2\right]<+\infty.
\]
By the Monotone Convergence Theorem, c.f. e.g.
\cite[Theorem~5.3]{Williams},
\[
    \mathbb{E}\left[\sum_{n=0}^{\infty}\sup_{a_n\leq t\leq a_{n+1}}|K(t)|^2\right]<+\infty.
\]
and hence
\[
    \sum_{n=0}^{\infty}\sup_{a_n\leq t\leq a_{n+1}}|K(t)|^2<+\infty, \quad \text{a.s.}
\]
Thus,
\[
    \lim_{n\to\infty}\sup_{a_n\leq t\leq a_{n+1}}|K(t)|^2 =0, \quad a.s.
\]
and therefore $\lim_{t\to\infty}K(t)=0$, a.s.
\end{proof}

\section{Proof of Proposition~\ref{prop:Y} and Corollary~\ref{cor:Yf}}
In this section, we give the proofs of the limiting behaviour of the
stochastic and deterministic convolutions which were stated in
Section~\ref{sect:asystochdetconv}.
\subsection{Proof of Proposition~\ref{prop:Y} for $\bf n\geq1$}
In the following $M$ denotes a positive constant whose value may
change from line to line.
        Using \eqref{eq:rR1} and \eqref{eq:Yvar} we may write
        \[
            Y(t) = \int_{0}^{t}S(t-s)\Sigma(s)\, dB(s) + \int_{0}^{t}R(t-s)\Sigma(s)\, dB(s), \quad t\geq0,
        \]
        where
        \[
            S(t) = \sum_{j=1}^{N}\e^{\alpha t}\{P_j(t)\cos(\beta_j t) + Q_j(t)\sin(\beta_j t)\} .
        \]
Thus,
\begin{equation}\label{eq:YSR}
    \frac{Y(t)}{t^n\e^{\alpha t}} = \int_{0}^{t}\frac{S(t-s)}{t^n\e^{\alpha t}}\Sigma(s)\, dB(s)
    + \int_{0}^{t}\frac{R(t-s)}{t^n\e^{\alpha t}}\Sigma(s)\, dB(s).
\end{equation}
We show using Theorem~\ref{thm.assuff2} that the second stochastic
integral term on the right--hand side above converges to zero almost
surely. So in the notation of Section~\ref{sect:stochlim} we define
\[
   H(t,s):= \frac{R(t-s)}{t^n\e^{\alpha t}}\Sigma(s).
\]
Now as $R(t)=O(t^{n-1}\e^{\alpha t})$ as $t\to\infty$ from
\eqref{eq:Rorder} it is natural to choose $H_{\infty}(s)=0_{n,d}$.
Thus we need only verify conditions \eqref{eq.Htilderateto02} and
\eqref{eq.H1to02}.
        Now, from \eqref{eq:Rorder} we have
        \begin{multline*}
            \int_{0}^{t}\norm{H(t,s)}_F^2 ds
            \\
\leq  \left(\frac{1+t}{t}\right)^{2n}\frac{1}{(1+t)^{2n}\e^{2\alpha
t}}\int_{0}^{t} M(1+t-s)^{2n-2} \e^{2\alpha(t-s)}\|\Sigma(s)\|^2_F\,
ds,
           \end{multline*}
           for some $M>0$.
           Hence for $t\geq 1$ we have
\begin{align*}
        \int_{0}^{t}\norm{H(t,s)}_F^2 ds    &\leq
 2^{2n}M\frac{1}{(1+t)^{2n}}\int_{0}^{t} (1+t-s)^{2n-2} \e^{-2\alpha s}\|\Sigma(s)\|^2_F\, ds\\
 &\leq
 2^{2n}M\frac{1}{(1+t)^{2}}\int_{0}^{t} \e^{-2\alpha s}\|\Sigma(s)\|^2_F\,
 ds\\
&\leq
 2^{2n}M\frac{1}{(1+t)^{2}}\int_{0}^\infty \e^{-2\alpha s}\|\Sigma(s)\|^2_F\,
 ds,
 \end{align*}
 where we use the fact that $\int_{0}^\infty e^{-2\alpha s}\|\Sigma(s)\|^2_F\,
 ds$ is finite.
    Therefore
        \[
            \lim_{t\to\infty}   \int_{0}^{t}\norm{H(t,s)}_F^2 ds \cdot \log t =0.
        \]
Next, we consider
\begin{align*}
\int_{k^\theta}^{(1+k)^\theta} \|H(s,s)\|^2_F\,ds &\leq
\int_{k^\theta}^{(1+k)^\theta} K  s^{-2n}\e^{-2\alpha
s}\norm{\Sigma(s)}_F^2\,ds \\
&\leq K k^{-2n\theta} \int_{k^\theta}^{(1+k)^\theta} \e^{-2\alpha
s}\norm{\Sigma(s)}_F^2\,ds,
\end{align*}
for some $K>0$. Since $n\geq 1$, $\theta>0$ and $\int_{0}^\infty
\e^{-2\alpha s}\|\Sigma(s)\|^2_F\,
 ds$ is finite, we have that
 \[
\lim_{k\to\infty} \int_{k^\theta}^{(1+k)^\theta} \|H(s,s)\|^2_F\,ds
\cdot \log k =0.
 \]
  Turning then to the derivative condition of \eqref{eq.H1to02} we see
        \begin{multline}\label{eq.dernm}
            H_1(t,s) = t^{-n}\e^{-\alpha t}R'(t-s)\Sigma(s) -\alpha\, t^{-n}\e^{-\alpha t}R(t-s)\Sigma(s) \\
            - n\,t^{-n-1}\e^{-\alpha t}R(t-s)\Sigma(s).
        \end{multline}
        Therefore we have
\begin{align*}
         \|H_1(t,s)\|_{F} &\leq  t^{-n} \e^{-\alpha t}\biggl( \|R'(t-s)\|_F +|\alpha|\, \|R(t-s)\|_F \\
            &\qquad + n t^{-1}\|R(t-s)\|_F  \biggr)\|\Sigma(s)\|_F,
\end{align*}
and so as $\|R(t)\|_F\leq M(1+t)^{n-1}\e^{\alpha t}$,
$\|R'(t)\|_F\leq M(1+t)^{n}\e^{\alpha t}$ we have for $t\geq 1$
\begin{align*}
\|H_1(t,s)\|_F &\leq M t^{-n} \e^{-\alpha s}
\biggl(  (1+t-s)^n +|\alpha| (1+t-s)^{n-1} \\
&\qquad+ n t^{-1}  (1+t-s)^{n-1}\biggr)\|\Sigma(s)\|_F\\
&\leq M t^{-n} (1+t-s)^n \left( 1 + (|\alpha|+n) (1+t-s)^{-1}\right)  \e^{-\alpha s} \| \Sigma(s)\|_F\\
&\leq M \left( 1 + |\alpha|+n \right) \cdot t^{-n} (1+t-s)^n
\e^{-\alpha s}\| \Sigma(s)\|_F.
\end{align*}
Thus for $t\geq 1$ we have
\begin{align*}
\int_0^t \|H_1(t,s)\|_F^2\,ds &\leq M_1^2 t^{-2n} \int_0^t (1+t-s)^{2n}  \e^{-2\alpha s}\| \Sigma(s)\|_F^2\,ds\\
 &\leq M_1^2 \left(\frac{1+t}{t}\right)^{2n} \int_0^t  \e^{-2\alpha s}\| \Sigma(s)\|_F^2\,ds\\
 &\leq M_1^2 2^{2n} \int_0^\infty  \e^{-2\alpha s}\| \Sigma(s)\|_F^2\,ds.
\end{align*}
Hence     $\int_{0}^{t}\norm{H_1(t,s)}_F^2ds $ may easily be bounded
above by a polynomially growing function. So we have shown that
        \begin{equation}  \label{eq:Rconv}
            \lim_{t\to\infty}\int_{0}^{t}\frac{R(t-s)}{t^n\e^{\alpha t}} \Sigma(s) \,dB(s) =0, \quad \text{a.s.}
        \end{equation}
    Next write
    \[
        P_j(t)=t^n P_{j}^* + P_{j,n-1}(t) \text{ and }  Q_j(t)=t^n Q_{j}^* + Q_{j,n-1}(t),
    \]
    where $P_{j,n-1}$ and $Q_{j,n-1}$ are matrix polynomials of order $n-1$.
    Then $S$ can be expressed according to
    \begin{align*}
            S(t) &= \sum_{j=1}^{N}\e^{\alpha t}t^n\{P_j^*\cos(\beta_j t) + Q_j^*\sin(\beta_j t)\}  \\
                        &\qquad +\sum_{j=1}^{N}\e^{\alpha t}\{P_{j,n-1}(t)\cos(\beta_j t) + Q_{j,n-1}(t)\sin(\beta_j t)\}.
    \end{align*}
    Thus,
    \begin{align}\label{eq:Sint}
        &\int_{0}^{t}\frac{S(t-s)}{t^n\e^{\alpha t}}\Sigma(s)\, dB(s) \\
        &\quad=\int_{0}^{t}\sum_{j=1}^{N}\e^{-\alpha s}\frac{(t-s)^n}{t^n}\{P_j^*\cos(\beta_j (t-s))
        + Q_j^*\sin(\beta_j (t-s))\} \Sigma(s)\, dB(s) \notag \\
        &\qquad +\int_{0}^{t}\sum_{j=1}^{N}\e^{-\alpha s} \frac{P_{j,n-1}(t-s)}{t^n}\cos(\beta_j (t-s)) \Sigma(s)\, dB(s) \notag \\
        &\qquad +\int_{0}^{t}\sum_{j=1}^{N}\e^{-\alpha s}\frac{Q_{j,n-1}(t-s)}{{t^n}}\sin(\beta_j (t-s))\Sigma(s)\, dB(s). \notag
    \end{align}
    We now argue that the second and third stochastic integrals on the right--hand side in \eqref{eq:Sint} tend to zero as $t\to\infty$.
    We focus on the second integral. Note that it suffices to show for any degree $n-1$ polynomial $P$ that
\[
\int_{0}^{t}  \frac{P(t-s)}{(1+t)^n}\cos(\beta (t-s))  \e^{-\alpha
s} \Sigma(s)\, dB(s)\to 0, \quad \text{ as } t\to\infty, \quad
\text{a.s.}
\]
By recalling the trigonometric identity, for any $a_1,a_2\in\R$,
    \begin{align}\label{eq:sumsin2}
        \cos(a_1-a_2) &= \cos(a_1)\cos(a_2)+\sin(a_1)\sin(a_2), \\
        \sin(a_1-a_2) &= \sin(a_1)\cos(a_2)-\cos(a_1)\sin(a_2), \notag
    \end{align}
    we see that it suffices to show that the process
\[
a(t)=\int_{0}^{t}  \frac{P(t-s)}{(1+t)^n} f(s)\,dB(s), 
\]
obeys $a(t)\to 0$ as $t\to\infty$ where $f$ is in
$L^2(\Rp;\R^{d\times d'})$ and $P$ is a matrix--valued polynomial of
degree $n-1$. Define $H(t,s)=P(t-s) (1+t)^{-n} f(s)$. Define
$H_\infty(s)=0$. Since $P$ is a polynomial, there exists $M$ such
that $|P(t)|\leq M(1+t)^{n-1}$ and $|P'(t)|\leq M(1+t)^{n-1}$ for
all $t\geq 0$.

Using Theorem~\ref{thm.assuff2} and the same procedure as used to
establish \eqref{eq:Rconv}, we get
    \[
        \lim_{t\to\infty}\int_{0}^{t}\sum_{j=1}^{N}\e^{-\alpha s}\frac{P_{j,n-1}(t-s)}{t^n}\cos(\beta_j (t-s)) \Sigma(s)\, dB(s)=0,
        \quad \text{a.s.}
    \]
    One can argue similarly that
    \[
        \lim_{t\to\infty}\int_{0}^{t}\sum_{j=1}^{N}\e^{-\alpha s}\frac{Q_{j,n-1}(t-s)}{t^n}\cos(\beta_j (t-s)) \Sigma(s)\, dB(s)=0,
         \quad \text{a.s.}
    \]
    We now turn our attention to the first integral term on the right--hand side of \eqref{eq:Sint}. Consider the integral
    \begin{equation} \label{def.A}
      A_j(t)=  \int_{0}^{t}\e^{-\alpha s}\frac{(t-s)^n}{t^n}P_j^*\cos(\beta_j (t-s))\Sigma(s)\, dB(s),
    \end{equation}
    and define
    \begin{multline*}
A_{j,0}(t)=  P_j^* \cos(\beta_j t) \int_{0}^{t}  \cos(\beta_j s) \e^{-\alpha s}\Sigma(s)\, dB(s) \\
+
 P_j^* \sin(\beta_j t)  \int_0^t \sin(\beta_j s) \e^{-\alpha s}\Sigma(s)\, dB(s).
    \end{multline*}
    Since $s\mapsto \e^{-\alpha s}\Sigma(s)$ is in $L^2(\Rp;\R^{d\times d'})$, if we define
    \begin{multline} \label{def.Aostar}
A_{j,0}^\ast(t)=  P_j^* \cos(\beta_j t) \int_{0}^\infty \cos(\beta_j
s) \e^{-\alpha s}\Sigma(s)\, dB(s) \\+
 P_j^* \sin(\beta_j t)  \int_0^\infty \sin(\beta_j s) \e^{-\alpha s}\Sigma(s)\, dB(s).
    \end{multline}
    we have that $A_{j,0}(t)-A_{j,0}^\ast(t)\to 0$ as $t\to\infty$ a.s.
By Newton's binomial expansion theorem
$(t-s)^n=\sum_{m=0}^{n}\binom{n}{m}t^{m}(-s)^{n-m}$ and using
\eqref{eq:sumsin2}, we get
  \begin{align*}
      A_j(t)&=
      \sum_{m=0}^n P_j^* (-1)^{n-m}  \binom{n}{m} \frac{1}{t^{n-m}} \int_{0}^{t} s^{n-m} \cos(\beta_j (t-s)) \e^{-\alpha s}\Sigma(s)\, dB(s) \\
    \end{align*}
    where we have defined for $k=1,\ldots,n$
    \[
A_{j,k}(t)= \frac{1}{t^k} \int_{0}^{t} s^k
 \left(\cos(\beta_j t) \cos(\beta_j s) + \sin(\beta_j t)\sin(\beta_j s)\right) \e^{-\alpha s}\Sigma(s)\, dB(s).
    \]
   This can be expressed as
      \begin{multline*}
A_{j,k}(t)=  \cos(\beta_j t) \frac{1}{t^k} \int_{0}^{t} s^k
\cos(\beta_j s)  \e^{-\alpha s}\Sigma(s)\, dB(s)
\\+
\sin(\beta_j t) \frac{1}{t^k} \int_{0}^{t} s^k \sin(\beta_j s)
\e^{-\alpha s}\Sigma(s)\, dB(s).
\end{multline*}
Now by applying Lemma~\ref{lemma.tminkskfto0} to each of the terms
on the righthand side, we get
\[
\lim_{t\to\infty} A_{j,k}(t)=0, \quad \text{a.s.}
\]
Therefore we see that
\begin{equation} \label{eq.Ajlim}
A_j(t)-A_{j,0}^\ast(t)
\to 0, \quad\text{as $t\to\infty$ a.s.}
\end{equation}

Define
\begin{equation} \label{def.Cj}
C_j(t)= \int_{0}^{t}\e^{-\alpha s}\frac{(t-s)^n}{t^n}
Q_j^*\sin(\beta_j (t-s)) \Sigma(s)\, dB(s)
\end{equation}
and
\[
C_{j,0}(t)=Q_j^* \int_{0}^{t}  \sin(\beta_j (t-s)) \e^{-\alpha
s}\Sigma(s)\, dB(s).
\]
Then
\begin{multline*}
C_{j,0}(t)=Q_j^* \sin(\beta_j t) \int_{0}^{t}
\cos(\beta_j s) \e^{-\alpha s}\Sigma(s)\, dB(s) \\
- Q_j^*\cos(\beta_j t)\int_0^t \sin(\beta_j s)\e^{-\alpha
s}\Sigma(s)\, dB(s),
\end{multline*}
and define
\begin{multline}  \label{def.Cjast}
C_{j,0}^\ast(t)=Q_j^* \sin(\beta_j t) \int_{0}^\infty
\cos(\beta_j s) \e^{-\alpha s}\Sigma(s)\, dB(s) \\
- Q_j^*\cos(\beta_j t)\int_0^\infty \sin(\beta_j s)\e^{-\alpha
s}\Sigma(s)\, dB(s).
\end{multline}
Then $C_{j,0}(t)-C_{j,0}^\ast(t)\to 0$ as $t\to\infty$ a.s., and by
proceeding as before we obtain
\begin{equation} \label{eq.Cjlim}
C_j(t)-C_{j,0}^\ast(t)
\to 0, \quad\text{as $t\to\infty$ a.s.}
\end{equation}

Therefore, returning to \eqref{eq:Sint} and using \eqref{def.A},
\eqref{def.Cj} we have
  \begin{align}\label{eq:Sint2}
        \lefteqn{\int_{0}^{t}\frac{S(t-s)}{t^n\e^{\alpha t}}\Sigma(s)\, dB(s) -\sum_{j=1}^N \{A_{j,0}^\ast(t) + C_{j,0}^\ast(t)\}} \\
        &=\sum_{j=1}^{N} \left\{ A_j(t) 
        - A_{j,0}^\ast(t)\right\}+
        \sum_{j=1}^N \left\{ C_j(t)
        -C_{j,0}^\ast(t)
        \right\}\notag \\
        &\quad +\int_{0}^{t}\sum_{j=1}^{N}\e^{-\alpha s} \frac{P_{j,n-1}(t-s)}{t^n}\cos(\beta_j (t-s)) \Sigma(s)\, dB(s) \notag \\
        &\quad +\int_{0}^{t}\sum_{j=1}^{N}\e^{-\alpha s}\frac{Q_{j,n-1}(t-s)}{{t^n}}\sin(\beta_j (t-s))\Sigma(s)\, dB(s), \notag
    \end{align}
    so by \eqref{eq.Ajlim} and \eqref{eq.Cjlim} we have
    \begin{equation} \label{eq:Sconv}
\lim_{t\to\infty} \left( \int_{0}^{t}\frac{S(t-s)}{t^n\e^{\alpha
t}}\Sigma(s)\, dB(s) -\sum_{j=1}^N \{A_{j,0}^\ast(t) +
C_{j,0}^\ast(t)\}\right)=0, \quad \text{a.s.}
    \end{equation}
    Using \eqref{eq:YSR}, \eqref{eq:Rconv}, \eqref{eq:Sconv} together with the definitions \eqref{def.Aostar} and \eqref{def.Cjast},
    we have
    \begin{align}\label{eq:Ycharexp}
        \lim_{t\to\infty}\left( \frac{Y(t)}{t^n\e^{\alpha t}}
        -\sum_{j=1}^{N} \{ \sin(\beta_j t)L_{1,j}+ \cos(\beta_j t)L_{2,j} \} \right)=0,
        \quad \text{a.s.}
    \end{align}
    where $L_{1,j}$ and $L_{2,j}$ are given by \eqref{def.L1} and \eqref{def.L2},
    which is \eqref{eq:asyY}.

\subsection{Proof of Proposition~\ref{prop:Y} for $\bf n=0$}
Using \eqref{eq:rR1} and \eqref{eq:Yvar} we may write
\[
    Y(t) = \int_{0}^{t}S(t-s)\Sigma(s)\, dB(s) + \int_{0}^{t}R(t-s)\Sigma(s)\, dB(s), \quad t\geq0,
\]
where
\[
    S(t) = \sum_{j=1}^{N}\e^{\alpha t}\{P_j^*\cos(\beta_j t) + Q_j^*\sin(\beta_j t)\} .
\]
Thus,
\begin{equation}\label{eq:YSR0}
    \e^{-\alpha t}Y(t) = \int_{0}^{t}\e^{-\alpha t}S(t-s)\Sigma(s)\, dB(s)
    + \int_{0}^{t}\e^{-\alpha t}R(t-s)\Sigma(s)\, dB(s).
\end{equation}
Defining $k(t)=\e^{-\alpha t}R(t)$, then from \eqref{eq:Rorder} and
\eqref{eq:R'order}, $k(t)= O(\e^{-\varepsilon t})$ and
\[
    |k'(t)| \leq |\alpha| |k(t)| + \e^{-\alpha t}|R'(t)| = O(\e^{-\varepsilon t})
\]
Thus
\[
    \int_{0}^{t}\e^{-\alpha t}R(t-s)\Sigma(s)\, dB(s) = \int_{0}^{t}k(t-s)\e^{-\alpha s}\Sigma(s)\, dB(s)
\]
and so Lemma~\ref{lm:kL2} applied element--wise gives
\begin{equation}\label{eq:Rto0}
    \lim_{t\to\infty}\int_{0}^{t}\e^{-\alpha t}R(t-s)\Sigma(s)\, dB(s)=0 \quad \text{a.s.}
\end{equation}
Moreover,
\begin{align}\label{eq:Sto0}
    \lim_{t\to\infty}\biggl(\int_{0}^{t}\e^{-\alpha t}&S(t-s)\Sigma(s)\, dB(s) \\
    & -\cos(\beta_j t)\int_{0}^{\infty}\e^{-\alpha s}\{P_j^*\cos(\beta_j s) - Q_j^*\sin(\beta_j s)\}\Sigma(s)\,dB(s) \notag \\
    &\quad -\sin(\beta_j t)\int_{0}^{\infty}\e^{-\alpha s}\{P_j^*\sin(\beta_j s) + Q_j^*\cos(\beta_j s)\}\Sigma(s)\,dB(s)\biggr)=0. \notag
\end{align}
Using \eqref{eq:Rto0} and \eqref{eq:Sto0} in \eqref{eq:YSR0}, gives
the required result.

\subsection{Proof of Corollary~\ref{cor:Yf}}
In order to prove Corollary~\ref{cor:Yf}, the following simple
asymptotic estimate is needed. It may be considered as a
deterministic analogue of Lemma~\ref{lemma.tminkskfto0}.
\begin{lemma}\label{lm:tlim}
For any $\phi\in L^1([0,\infty);\R^{d})$,
\[
    \lim_{t\to\infty}\frac{1}{t^j}\int_{0}^{t}s^j \phi(s)\,ds =0, \quad j=1,...,n.
\]
\end{lemma}
\begin{proof}
For any $\theta\in(0,1)$,
\begin{align*}
    \left|\frac{1}{t^j}\int_{0}^{t}s^j \phi(s)ds\right| &\leq \frac{1}{t^j}\int_{0}^{\theta t}s^j |\phi(s)|ds
        + \frac{1}{t^j}\int_{\theta t}^{t}s^j |\phi(s)|ds \\
        &\leq \theta^j\int_{0}^{\infty}|\phi(s)|ds + \int_{\theta t}^{\infty}|\phi(s)|ds
\end{align*}
Thus,
\[
    \limsup_{t\to\infty}\left|\frac{1}{t^j}\int_{0}^{t}s^j \phi(s)ds\right| \leq \theta^j\int_{0}^{\infty}|\phi(s)|ds.
\]
Letting $\theta\to0$ gives the result.
\end{proof}

We are now in a position to proceed with the proof of
Corollary~\ref{cor:Yf}. Firstly consider the case $n\geq1$. The
asymptotic behaviour of $Y$ is known from Proposition~\ref{prop:Y}.
Thus we concentrate solely upon the term
$\int_{0}^{t}r(t-s)f(s)\,ds$ in \eqref{eq:VpY} in determining the
asymptotic behaviour of $V$.
Defining
\[
    S(t) = \sum_{j=1}^{N}\e^{\alpha t}\{P_j(t)\cos(\beta_j t) + Q_j(t)\sin(\beta_j t)\}, \quad t\geq0.
\]
Then we have
\[
    \int_{0}^{t}\frac{r(t-s)}{t^n \e^{\alpha t}}f(s)\,ds
    = \int_{0}^{t}\frac{S(t-s)}{t^n \e^{\alpha t}}f(s)\,ds + \int_{0}^{t}\frac{R(t-s)}{t^n \e^{\alpha t}}f(s)\,ds.
\]
Then,
\begin{align*}
    \left|\int_{0}^{t}\frac{R(t-s)}{t^n \e^{\alpha t}}f(s)\,ds\right|
    &\leq \frac{1}{(1+t)^n}M\int_{0}^{t}(1+t-s)^{n-1}\e^{-\alpha s}|f(s)|ds \\
    &\leq \frac{1}{1+t}M\int_{0}^{t}\e^{-\alpha s}|f(s)|ds.
\end{align*}
Taking the limit superior, as $t\to\infty$, over this inequality
yields,
\[
    \lim_{t\to\infty}\int_{0}^{t}\frac{R(t-s)}{t^n \e^{\alpha t}}f(s)\,ds =0.
\]
In analysing the term $S(t-s)$ one may decompose the trigonometric
terms via \eqref{eq:sumsin2}, whilst the polynomial terms, $P_j$ and
$Q_j$ may be dealt with using Newton's binomial expansion, i.e.
\[
    (t-s)^n = \sum_{m=0}^{n} \binom{n}{m} t^{n-m}(-s)^m.
\]
This, together with Lemma~\ref{lm:tlim}, yields
\[
    \lim_{t\to\infty}\left(\int_{0}^{t}\frac{r(t-s)}{t^n \e^{\alpha t}}f(s)\,ds
    -\sum_{j=1}^{N}\{\sin(\beta_j t)D_{1,j} + \cos(\beta_j t)D_{2,j}\} \right) =0.
\]
with
\begin{align*}
D_{1,j} &= \int_{0}^{\infty}\e^{-\alpha s}\{P_j^*\sin(\beta_j s) + Q_j^*\cos(\beta_j s)\}  f(s)\,ds, \\
D_{2,j} &= \int_{0}^{\infty}\e^{-\alpha s}\{P_j^*\cos(\beta_j s) -
Q_j^*\sin(\beta_j s)\}  f(s)\,ds.
\end{align*}
Combining this with Proposition~\ref{prop:Y} yields the result for
$V$.

For the case $n=0$, the proof follows as for the case $n\geq1$.
However in the analysis of the remainder term, $R$, it is required
to understand the asymptotic behaviour of the integral
\[
    \int_{0}^{t}\e^{-\varepsilon(t-s)} \e^{-\alpha s}f(s)\,ds.
\]
This integral is the convolution of a term in $L^1(0,\infty)$ with a
term which tends to zero. Hence this integral itself tends to zero,
\cite[Theorem~2.2.2~(i)]{GrLoSt90}.

\section{Proof of Lemmas~\ref{lm:Rvol} and~\ref{lm:Rfin}} \label{sect:pflemmaRRprasy}
This section contains the asymptotic estimates needed for the
remainder terms $R$ defined in \eqref{eq:rRvol} and
\eqref{eq:rRsfde}.
\subsection{Proof of Lemma~\ref{lm:Rvol}}
We start with the proof of a preliminary lemma.
\begin{lemma}\label{lm:Keig}
    Let $K_{j,0}$ be defined by \eqref{eq:tayK} with $n=0$. Then
    \[
        \left(\lambda_j I_d - \int_{[0,\infty)}\e^{-\lambda_j s}\mu(ds) \right)K_{j,0} =0_{d,d},
    \]
    where $\lambda_j\in\Lambda'$ are zeroes of $h_{\mu}(\lambda)$.
\end{lemma}
A corresponding result can be shown for the zeroes of the
characteristic equation, $g_\nu$, of the finite delay equation using
\eqref{eq:tayfK} and is omitted.
\begin{proof}[Proof of Lemma~\ref{lm:Keig}]
Multiply \eqref{eq:tayK} on the left by
$(\lambda-\lambda_j)\left(\lambda I_d -
\int_{[0,\infty)}\e^{-\lambda s}\mu(ds) \right)$ to get
\begin{multline*}
    (\lambda-\lambda_j)I_d = \left(\lambda I_d - \int_{[0,\infty)}\e^{-\lambda s}\mu(ds) \right)K_{j,0}
        \\+ (\lambda-\lambda_j)\left(\lambda I_d - \int_{[0,\infty)}\e^{-\lambda s}\mu(ds) \right) \hat{q}_j(\lambda).
\end{multline*}
Now let $\lambda\to\lambda_j$, recalling that $\hat{q}_j(\lambda)$
is analytic at $\lambda_j$, to get the result.
\end{proof}

We are now in a position to prove Lemma~\ref{lm:Rvol}. To start, we
define $\tilde{q}(t)=\e^{-\alpha t}q(t)$ for $t\geq 0$. Then
$\tilde{q}$ is differentiable a.e. and
\[
\int_0^\infty \e^{\varepsilon t}|\tilde{q}(t)|\,dt <+\infty,
\]
where $\varepsilon$ is defined as in Subsection~\ref{sub.Vol}. Also
$|\tilde{q}'(t)|\leq \e^{-\alpha t}|q'(t)|+|\alpha| \e^{-\alpha
t}|q(t)|$ for $t\geq 0$. Since $q, q'\in
L^1(\mathbb{R}^+;\varphi;\mathbb{R}^{d\times d})$, we have
\[
 \int_0^\infty \e^{\varepsilon t}|\tilde{q}'(t)|\,dt
 \leq \int_0^\infty  \e^{\varepsilon t}\e^{-\alpha t}|q'(t)|\,dt +\int_0^\infty |\alpha|\,
\e^{\varepsilon t}\e^{-\alpha t}|q(t)|\,dt<+\infty.
\]
Finally, we have that
\[
\tilde{q}(t)\e^{\varepsilon t}=\tilde{q}(0) + \int_0^t
\tilde{q}'(s)\e^{\varepsilon s}\,ds+ \varepsilon\int_0^t
\tilde{q}(s)\e^{\varepsilon s}\,ds,
\]
so $|\tilde{q}(t)|\leq C\e^{-\varepsilon t}$ for all $t\geq 0$.

Let $\Lambda_n'=\{\lambda_1,...,\lambda_N\}$. Then from
\eqref{eq.solsumrepvol} and \eqref{eq:rRvol}, we get
\begin{align*}
\e&^{-\alpha t}R(t) \\
&= \sum_{\lambda_j\in \Lambda_\varepsilon \setminus \Lambda'_n,
\Im(\lambda_j)\geq0} \e^{-(\alpha-\Re(\lambda_j)) t}
\{P_j(t)\cos(\Im(\lambda_j)t)+Q_j(t)\sin(\Im(\lambda_j) t)\} +
\tilde{q}(t)\\
&= \sum_{\lambda_j\in \Lambda' \setminus \Lambda'_n,
\Im(\lambda_j)\geq0} \e^{-(\alpha-\Re(\lambda_j)) t}
\{P_j(t)\cos(\Im(\lambda_j)t)+Q_j(t)\sin(\Im(\lambda_j) t)\} \\
&\quad+
 \sum_{\lambda_j\in \Lambda_\varepsilon\setminus \Lambda', \Im(\lambda_j)\geq0}
\e^{-(\alpha-\Re(\lambda_j)) t}
\{P_j(t)\cos(\Im(\lambda_j)t)+Q_j(t)\sin(\Im(\lambda_j) t)\} +
\tilde{q}(t).
\end{align*}
If $n=0$, then $R(t)=O(\e^{(\alpha-\varepsilon)t})$ as $t\to\infty$.
If $n\geq 1$, and $\Lambda_n'=\Lambda' \cap\{\Im(\lambda)\geq0\}$,
then $R(t)=O(\e^{(\alpha-\varepsilon) t})$. If $n\geq 1$, and
$\Lambda_n'\subset \Lambda' \cap\{\Im(\lambda)\geq0\}$, then
$R(t)=O(t^{n-1}\e^{\alpha t})$ as $t\to\infty$. Therefore if $n\geq
1$, we always have $R(t)=O(t^{n-1}\e^{\alpha t})$ as $t\to\infty$.

We now prove the estimate on the derivative. We deal here with the
case $n\geq1$.
    From \eqref{eq.fund1} we know that $r$ is differentiable and hence from \eqref{eq:rRvol} so too is $R$. Defining
    \[
        S(t):= \sum_{j=1}^{N}\e^{\alpha t}\{P_j(t)\cos(\beta_j t) + Q_j(t)\sin(\beta_j t)\}
    \]
    and using \eqref{eq.fund1} and \eqref{eq:rRvol} we have
    \begin{align*}
        R'(t) &=  r'(t) - S'(t)
        =  \int_{[0,t]}  \mu(ds)\, r(t-s) - S'(t).
    \end{align*}
It is clear from \eqref{eq.solsumrepvol} that
$r(t)=O(t^{n}\e^{\alpha t})$ and from the definition of $S$ that
$S'(t)=O(t^{n}\e^{\alpha t})$. Therefore, it follows that
$\norm{r(t)}\leq M(1+t)^{n}\e^{\alpha t}$ and $\norm{S'(t)}\leq
M(1+t)^{n}\e^{\alpha t}$ for $t\geq 0$ and some $M>0$. Hence as
$|\mu|\in M(\mathbb{R}_+;\mathbb{R})$ and $\int_{[0,\infty)}
\e^{-\alpha s}|\mu|(ds)<+\infty$, we have
 \begin{align*}
       \norm{R'(t)} &\leq \int_{[0,t]}  |\mu|(ds)\, \norm{r(t-s)} +  \norm{S'(t)} \\
   &\leq
   \int_{[0,t]}  |\mu|(ds)\, M(1+t-s)^{n}\e^{\alpha (t-s)} + M(1+t)^{n}\e^{\alpha t}\\
   &\leq
   \int_{[0,t]}  |\mu|(ds)\, M(1+t)^{n}\e^{\alpha (t-s)} + M(1+t)^{n}\e^{\alpha t}\\
   &\leq M(1+t)^{n} \e^{\alpha t}   \int_{[0,\infty)}  \e^{-\alpha s}|\mu|(ds)  + M(1+t)^{n}\e^{\alpha
   t},
    \end{align*}
and therefore  $R'(t) = O(t^n \e^{\alpha t})$ for $n\geq1$.

For the case $n=0$, we define
\[
    S(t):= \sum_{j=1}^{N}\e^{\alpha t}\{P_j^*\cos(\beta_j t) + Q_j^*\sin(\beta_j t)\},
\]
then the real function $S$ can be rewritten concisely using complex
constants as
\[
    S(t) = \sum_{\lambda_j\in\Lambda'} \e^{\lambda_j t} K_{j,0}.
\]
As $R(t) = r(t) - S(t)$ we have
\begin{align*}
    R'(t) &= r'(t) - S'(t)
    = \int_{[0,t]}\mu(ds)r(t-s) - \lambda_j\sum_{\lambda_j\in\Lambda'} \e^{\lambda_j t} K_{j,0} \\
    &= \int_{[0,t]}\mu(ds)R(t-s) + \int_{[0,t]}\mu(ds)\sum_{\lambda_j\in\Lambda'} \e^{\lambda_j (t-s)} K_{j,0}
        - \sum_{\lambda_j\in\Lambda'}\lambda_j\,\e^{\lambda_j t} K_{j,0} \\
    &= \int_{[0,t]}\mu(ds)R(t-s) - \sum_{\lambda_j\in\Lambda'}\e^{\lambda_j t}
    \left(\lambda_j\,I_d - \int_{[0,t]}\e^{-\lambda_j s}\mu(ds)\right) K_{j,0} \\
    & =  \int_{[0,t]}\mu(ds)R(t-s) - \sum_{\lambda_j\in\Lambda'}\e^{\lambda_j t}
    \left(\lambda_j\,I_d - \int_{[0,\infty)}\e^{-\lambda_j s}\mu(ds)\right) K_{j,0} \\
    &\qquad -  \sum_{\lambda_j\in\Lambda'}\e^{\lambda_j t}\int_{(t,\infty)}\e^{-\lambda_j s}\mu(ds) K_{j,0}.
\end{align*}
By Lemma~\ref{lm:Keig} the second term on the right--hand side is
equal to zero, and so
\begin{equation}\label{eq:R'O}
|R'(t)| \leq \left|\int_{[0,t]}\mu(ds)R(t-s)\right| +
\left|\sum_{\lambda_j\in\Lambda'}\e^{\lambda_j
t}\int_{(t,\infty)}\e^{-\lambda_j s}\mu(ds) K_{j,0}\right|.
\end{equation}
Now,
\begin{align*}
    \left|\int_{[0,t]}\mu(ds)R(t-s)\right| &\leq \int_{[0,t]}|\mu|(ds)M \e^{(\alpha-\varepsilon)(t-s)} \\
    &= \e^{(\alpha-\varepsilon)t} \int_{[0,t]}\e^{-(\alpha-\varepsilon)s}|\mu|(ds)M \\
    &\leq \e^{(\alpha-\varepsilon)t} \int_{[0,\infty)}\e^{-(\alpha-\varepsilon)s}|\mu|(ds)M.
\end{align*}
Thus, $\int_{[0,t]}\mu(ds)R(t-s) = O(\e^{(\alpha-\varepsilon)t})$.
Recalling that $\lambda_j=\alpha+i\beta_j$ and so $|\e^{\lambda_j
t}|=\e^{\alpha t}$. Thus,
\begin{align*}
        \left|\sum_{\lambda_j\in\Lambda'}\e^{\lambda_j t}\int_{(t,\infty)}\e^{-\lambda_j s}\mu(ds) K_{j,0}\right|
        &\leq \e^{\alpha t}\sum_{\lambda_j\in\Lambda'}\int_{(t,\infty)}\e^{-\alpha s}|\mu|(ds)M \\
        &= \e^{\alpha t}\sum_{\lambda_j\in\Lambda'}\int_{(t,\infty)}\e^{-\varepsilon s}\e^{-(\alpha-\varepsilon) s}|\mu|(ds)M \\
        &\leq \e^{(\alpha-\varepsilon) t}\sum_{\lambda_j\in\Lambda'}\int_{(t,\infty)}\e^{-(\alpha-\varepsilon) s}|\mu|(ds)M \\
        &\leq \e^{(\alpha-\varepsilon) t} M_1,
\end{align*}
where it is noted that $\Lambda'$ contains finitely many elements.
Therefore, \eqref{eq:R'O} gives
\[
    R'(t) = O(\e^{(\alpha-\varepsilon)t}), \quad t\to\infty,
\]
and this completes the proof.

\subsection{Proof of Lemma~\ref{lm:Rfin}}
We now  use \eqref{eq.solsumrep} to determine properties of $R$ of
\eqref{eq:rRsfde}. From \eqref{eq:rRsfde}
\begin{equation*} 
    r(t) = \sum_{j=1}^{N}\e^{\alpha t}\{P_j(t)\cos(\beta_j t) + Q_j(t)\sin(\beta_j t)\} + R(t), \quad
    t\geq0.
\end{equation*}

In the case when $\{\lambda_1,...,\lambda_N\}=\Lambda'
\cap\{\Im(\lambda)\geq0\}$, we have that
$R(t)=e^{(\alpha-\varepsilon) t}$ for all $\varepsilon\in
(0,\varepsilon_0)$. If $n\geq 1$, and
$\{\lambda_1,...,\lambda_N\}\subset \Lambda'
\cap\{\Im(\lambda)\geq0\}$, then $R(t)=O(t^{n-1}e^{\alpha t})$ as
$t\to\infty$. Therefore if $n\geq 1$, we always have
$R(t)=O(t^{n-1}e^{\alpha t})$ as $t\to\infty$. If $n=0$, then
$R(t)=O(e^{(\alpha-\varepsilon)t})$ as $t\to\infty$.

We deal here with the case $n\geq1$.
From \eqref{eq.fund22} we
    know that $r$ is differentiable and hence from \eqref{eq:rRsfde} so too is $R$. Defining
    \[
        S(t):= \sum_{j=1}^{N}\e^{\alpha t}\{P_j(t)\cos(\beta_j t) + Q_j(t)\sin(\beta_j t)\}
    \]
    and using \eqref{eq.fund22} and \eqref{eq:rRsfde} we have
    \begin{align*}
        R'(t) &=  r'(t) - S'(t)
        =  \int_{[-\tau,0]}  \nu(ds)\, r(t+s) - S'(t), \quad \text{ for all } t\geq \tau.
    \end{align*}
It is clear from \eqref{eq.solsumrep} that $r(t)=O(t^{n}\e^{\alpha
t})$ and from the definition of $S$ that $S'(t)=O(t^{n}\e^{\alpha
t})$. Thus, there exists $t_0\geq0$ and positive constant matrices
$M_1,M_2$ such that for $t\geq t_0+\tau$,
    \begin{align*}
        |R'(t)| &\leq  \int_{[-\tau,0]}  |\nu|(ds)\, |r(t+s)| + t^{n}\e^{\alpha t}M_2 \\
        & \leq \int_{[-\tau,0]}  |\nu|(ds)\, (s+t)^{n}\e^{\alpha (t+s)} M_1 + t^{n}\e^{\alpha t}M_2 \\
        & \leq t^{n}\e^{\alpha t}\int_{[-\tau,0]}  \e^{\alpha s}\,|\nu|(ds) M_1 + t^{n}\e^{\alpha t}M_2.
    \end{align*}
Thus, $R'(t) = O(t^{n}\e^{\alpha t})$.

What remains to be covered is the case when $n=0$. To do this, we
start by making the observation
that the differential resolvent of \eqref{eq.fund22} may be regarded
as the solution of a Volterra equation.
To see this, define $\nu_{+}(E)=\nu(-E)$ where $-E=\{x:-x\in E\}$
for all
 sets $E$ which are subsets of the Borel sets formed from the interval $[0,\tau]$
 and $\nu_+(E)=0$ for all
sets $E$ which are subsets of the Borel sets formed from the
interval $(\tau,\infty)$. Then
\[
    r'(t) = \int_{[0,\tau]}\nu_+(ds) r(t-s) \text{ for } t\geq0, \quad r(0)=I_d.
\]
For $t>\tau$,
\begin{align*}
    r'(t) &= \int_{[0,t]}\nu_+(ds)r(t-s) - \int_{(\tau,t]}\nu_+(ds)r(t-s) \\
    &= \int_{[0,t]}\nu_+(ds)r(t-s)
\end{align*}
as $\nu_+=0$ in the second term on the right--hand side. On the
other hand, for $0\leq t\leq\tau$, it is true that $\max\{-\tau,-t\}
=-t$ and hence
\[
    r'(t) = \int_{[0,t]}\nu_+(ds) r(t-s) \quad \text{ for } t\geq0, \quad r(0)=I_d.
\]
The case $n=0$ follows from this observation, using a
similar proof to that which established Lemma~\ref{lm:Rvol}.

\section{Proof of Remark~\ref{rk:fsignec}}   \label{sect:remarkpf}
In this case $r(t)=e^{\alpha t}$ and $X$ obeys, for $t\geq0$,
\begin{equation}\label{eq:scXf}
 e^{-\alpha t}X(t)=X_0 + \int_{0}^{t}\e^{-\alpha s}f(s)\,ds + \int_0^t e^{-\alpha s}\Sigma(s)\,dB(s).
\end{equation}
Define the Gaussian martingale $M$ by $ M(t)=\int_0^t e^{-\alpha
s}\Sigma(s)\,dB(s)$ and the deterministic function $d$ by
$d(t)=X_0+\int_{0}^{t}\e^{-\alpha s}f(s)\,ds$. Then from
\eqref{eq:Mconv}, we have on this event of positive probability that
\[
\lim_{t\to\infty} \{M(t)+d(t)\} = L\in(-\infty,\infty).
\]

Suppose that $\lim_{t\to\infty}\langle M \rangle(t)=+\infty$.
Consequently $\limsup_{t\to\infty}M(t)=+\infty$ and
$\liminf_{t\to\infty}M(t)=-\infty$. Also,
$\limsup_{t\to\infty}d(t)=+\infty$, otherwise, if $d(t)\leq D$ for
all $t\geq0$, we have
\[
    L= \liminf_{t\to\infty}\{ d(t) + M(t)\}\leq D + \liminf_{t\to\infty}M(t) = -\infty,
\]
which is a contradiction. (Similarly one can show that
$\liminf_{t\to\infty}d(t)=-\infty$).

Then there exists a deterministic sequence
$\{t_n\}_{n\in\mathbb{Z}^+}$, with $t_0=0$ and $t_n\to\infty$ as
$n\to\infty$, such that $d(t_{n+1})>d(t_n)$ and $d(t_n)\to\infty$ as
$n\to\infty$. Then $M(t_n)\to-\infty$ as $n\to\infty$.

Now,
\[
    \tilde{M}(n) := M(t_n) = \sum_{j=1}^{n}\int_{t_{j-1}}^{t_j}\e^{-\alpha s}\Sigma(s)\,dB(s) = \sum_{j=1}^{n}G_j,
\]
where each $G_j=\int_{t_{j-1}}^{t_j}\e^{-\alpha s}\Sigma(s)\,dB(s)$
is a Gaussian distributed random variable with mean zero and
variance $\int_{t_{j-1}}^{t_j}\e^{-2\alpha t}\Sigma(t)^2\,ds$, each
$G_j$ is measurable with respect to the filtration
$\mathcal{G}_n=\mathcal{F}^B(t_n)$, $n\geq1$,
and $\{G_j\}_{j\in\mathbb{Z}^+}$ are independent and $\langle
\tilde{M} \rangle (n) = \langle M \rangle
(t_n)=\int_{0}^{t_n}\e^{-2\alpha t}\Sigma(t)^2\,ds\to\infty$ as
$n\to\infty$.

Therefore by arguments akin to that used in Shiryeav
\cite[Section~4.1]{Shir}
\begin{align}\label{eq:Shiras}
    \mathbb{P}\left[\limsup_{n\to\infty}\frac{\tilde{M}(n)}{\sqrt{\langle\tilde{M}\rangle(n)}}=+\infty\right] =1, \quad
    \mathbb{P}\left[\liminf_{n\to\infty}\frac{\tilde{M}(n)}{\sqrt{\langle\tilde{M}\rangle(n)}}=-\infty\right] =1,
\end{align}
which implies that
$\mathbb{P}\left[\limsup_{n\to\infty}\tilde{M}(n)=+\infty\right] =1$
and so that
\[
    \mathbb{P}\left[\limsup_{n\to\infty}M(t_n)=+\infty\right] =1.
\]
But our assumption gave that $\lim_{n\to\infty}M(t_n)=-\infty$, with
positive probability. Thus a contradiction. Hence $\langle M \rangle
(t)\to L'\in(-\infty,\infty)$ as $t\to\infty$, i.e.
\[
    \int_{0}^{\infty}\e^{-2\alpha t}\Sigma(t)^2\,dt<+\infty.
\]
Therefore $M(t)\to M(\infty)\in(-\infty,\infty)$ as $t\to\infty$
a.s. and so $\lim_{t\to\infty}d(t)=\lim_{t\to\infty}\{
d(t)+M(t)-M(t) \} =L-M(\infty)\in(-\infty,\infty)$. Hence
\[
    \lim_{t\to\infty}\int_{0}^{t}\e^{-\alpha s}f(s)\,ds =\lim_{t\to\infty}\{ d(t) -X_0 \}\in(-\infty,\infty).
\]

All that remains to be shown is the validity of \eqref{eq:Shiras},
i.e. we need to show that
\[
    A'= \left\{ \limsup_{n\to\infty}\frac{\tilde{M}(n)}{\sqrt{\langle \tilde{M} \rangle(n)}}=+\infty\right\},
    \quad A''= \left\{\liminf_{n\to\infty}\frac{\tilde{M}(n)}{\sqrt{\langle \tilde{M} \rangle(n)}}=-\infty \right\}
\]
are almost sure events. Let
\[
    A_c'= \left\{ \limsup_{n\to\infty}\frac{\tilde{M}(n)}{\sqrt{\langle \tilde{M} \rangle(n)}} >c \right\},
    \quad A_c''= \left\{ \liminf_{n\to\infty}\frac{\tilde{M}(n)}{\sqrt{\langle \tilde{M} \rangle(n)}} <-c \right\}.
\]
Then $A_c'\to A'$ and $A_c''\to A''$ as $c\to\infty$ and
$A',A'',A_c',A_c''$ are tail events. We show that
$\mathbb{P}[A_c']=\mathbb{P}[A_c'']=1$ for all $c>0$.

Using Section~4.1.5 Probem~5, pp.383 of \cite{Shir} gives
\[
    \mathbb{P}[A_c'] = \mathbb{P}\left[\limsup_{n\to\infty}\frac{\tilde{M}(n)}{\sqrt{\langle \tilde{M} \rangle(n)}}>c\right]
    \geq \limsup_{n\to\infty}\mathbb{P}\left[\frac{\tilde{M}(n)}{\sqrt{\langle \tilde{M} \rangle(n)}}>c\right] = 1-\Phi(c)>0
\]
and
\[
    \mathbb{P}[A_c''] = \mathbb{P}\left[\liminf_{n\to\infty}\frac{\tilde{M}(n)}{\sqrt{\langle \tilde{M} \rangle(n)}}<-c\right]
    = \mathbb{P}\left[\limsup_{n\to\infty}\frac{-\tilde{M}(n)}{\sqrt{\langle \tilde{M} \rangle(n)}}>c\right]
     \geq 1-\Phi(c)>0.
\]
So, $\mathbb{P}[A_c']>0$ and $\mathbb{P}[A_c'']>0$, then since the
$G_j$'s are independent an application of Kolomogrov's Zero-One Law,
c.f. e.g. \cite[Theorem~4.1.1]{Shir}, implies
$\mathbb{P}[A_c']=\mathbb{P}[A_c'']=1$. Therefore
$\mathbb{P}[A']=\lim_{c\to\infty}\mathbb{P}[A_c']=1$ and
$\mathbb{P}[A'']=\lim_{c\to\infty}\mathbb{P}[A_c'']=1$.

\end{document}